\documentclass[12pt,reqno]{amsart}
\usepackage{a4wide}

\numberwithin{equation}{section}
\usepackage{mathrsfs}
\usepackage{amsfonts}
\usepackage{amsmath}
\usepackage{stmaryrd}
\usepackage{amssymb}
\usepackage{amsthm}
\usepackage{mathrsfs}
\usepackage{url}
\usepackage{amsfonts}
\usepackage{amscd}
\usepackage{indentfirst}
\usepackage{enumerate}
\usepackage{amsmath,amsfonts,amssymb,amsthm}
\usepackage{amsmath,amssymb,amsthm,amscd}
\usepackage{graphicx,mathrsfs}
\usepackage{appendix}

\usepackage[numbers,sort&compress]{natbib}

\usepackage{color}
 \usepackage[colorlinks, linkcolor=blue, citecolor=blue]{hyperref}

\newcommand{\R}{\mathbb{R}}

\renewcommand{\theequation}{\arabic{section}.\arabic{equation}}

\newtheorem{Thm}{Theorem}[section]
\newtheorem{Lem}[Thm]{Lemma}

\newtheorem{Prop}[Thm]{Proposition}

\begin{document}

\title[Excited states on BEC problem]
{Excited states on Bose-Einstein condensates with attractive interactions}
\author[P. Luo, S. Peng, S. Yan]{Peng Luo, Shuangjie Peng and Shusen Yan}

\address[Peng Luo]{School of Mathematics and Statistics, Central China Normal University, Wuhan 430079, China}
\email{pluo@mail.ccnu.edu.cn}

\address[Shuangjie Peng]{School of Mathematics and Statistics and Hubei Key Laboratory of Mathematical Sciences, Central China Normal University, Wuhan 430079, China}
\email{sjpeng@mail.ccnu.edu.cn}

\address[Shusen Yan]{School of Mathematics and Statistics, Central China Normal University, Wuhan 430079, China}
\email{syan@mail.ccnu.edu.cn}

\thanks{Luo and Peng were supported  by the Key Project of NSFC (No.11831009). Luo  was partially supported by NSFC grants (No.11701204).
 Yan was supported by  NSFC grants (No.11629101).}


\date{\today}

\begin{abstract}
We study the Bose-Einstein condensates (BEC) in two or three dimensions  with attractive interactions, described by $L^{2}$ constraint Gross-Pitaevskii energy functional.
First, we give the precise description of  the chemical potential of the condensate $\mu$ and the attractive interaction $a$. Next, for a class of degenerated trapping potential with non-isolated critical points,
 we obtain the existence and the local uniqueness of excited states  by precise analysis of the concentrated points and the Lagrange multiplier.

To our best knowledge, this is the first result concerning on excited states of BEC in Mathematics.
 Also, our results show that $ka_*$ are critical values in  two dimension when the concentration occurs for any positive integer $k$ with some positive constant $a_*$. And we point out that our results on degenerated trapping potential with non-isolated critical points are also new even for the classical Schr\"odinger equations.
 Here our main tools are  finite-dimensional reduction  and
  various Pohozave identities. The main difficulties come from the
  estimates on Lagrange multiplier and  the different degenerate rate along different directions at the critical points of $V(x)$.
\end{abstract}

\maketitle

{\small

\section{Introduction}

\setcounter{equation}{0}

The idea of Bose-Einstein condensates (BEC) originated in 1924-1925, when  Einstein predicted that, below a critical temperature, part of the bosons would occupy the
same quantum state to form a condensate.  Over the last two decades,
remarkable experiments on BEC in dilute gases of alkali atoms  \cite{Anderson,Bloch,Davis}  have revealed various interesting quantum phenomena. These new experimental advances make  many mathematicians  study again  the following of Gross-Pitaevskii (GP) equations proposed by Gross \cite{Gross} and Pitaevskii \cite{Pitaevskii} in the
 1960s:
\begin{equation}\label{1-22-5}
i \partial_t \psi(x, t)= -\Delta \psi(x, t) +V(x) \psi(x, t) -a |\psi(x, t)|^2\psi(x, t),\quad x\in \mathbb R^N,
\end{equation}
with the constraint
 \begin{equation*}
\int_{\mathbb R^N} |\psi(x, t)|^2\,dx=1,
\end{equation*}
where $N=2, 3$, $V(x)\ge 0$ is a real-valued potential and $a\in \mathbb R$  is treated as an arbitrary
dimensionless parameter. For better
understanding of the long history and further results on Bose-Einstein condensates,  we refer to \cite{Cornell,Lieb1,Lieb2,Ketterle} and   the references therein.

If we want to find a solution for \eqref{1-22-5} of the form $
\psi(x, t)= u(x) e^{-i \mu t}$,
where $\mu$ represents  the chemical potential of the condensate and $u(x)$ is a function independent
of time, then the unknown pair $(\mu, u)$ satisfies the following nonlinear eigenvalue equation
\begin{equation}\label{8-18-1}
-\Delta u +V(x) u = a u^3 +\mu u,\quad \text{in}\;\mathbb R^N,
\end{equation}
and the following constraint
\begin{equation}\label{8-28-3}
\int_{\mathbb R^N} u^2=1.
\end{equation}
The   energy functional corresponding to \eqref{8-18-1} is given by
\begin{equation}\label{5-6-1}
J(u) = \int_{\mathbb R^N}  \bigl( |\nabla u|^2 +V(x) u^2 \bigr) -\frac a4 \int_{\mathbb R^N} u^4.
\end{equation}
A ground state solution of \eqref{8-18-1} is a minimizer of the following minimization problem:
\begin{equation}\label{1-7-11}
I_a:= \inf\Bigl\{ \frac12 \int_{\mathbb R^N} \bigl( |\nabla u|^2 + V(x) u^2 -\frac a4 \int_{\mathbb R^N}  u^4:\; u\in H^1(\mathbb R^N),\; \int_{\mathbb R^N} u^2=1\Bigr\}.
\end{equation}
Any eigenfunction of \eqref{8-18-1}  whose energy is larger than that of
the ground state is usually called the excited states in the physics literatures in \cite{Bao}.

Let  us first recall the existence result for the ground state.
 Denote by  $Q(x)$   the unique positive solution of
$
-\Delta u+u=u^3, ~u\in H^1(\R^N)$ with $N=2,3$.  Let
$a_*=\displaystyle\int_{\R^N}Q^2$.  We have

\noindent {\bf Theorem~A.}(c.f.\cite{Bao})  { \it Suppose $V (x)\ge 0 \; (x \in \mathbb R^N) $ satisfies
\[
\lim_{|x|\to +\infty} V(x)=+\infty.
\]
If (i)  $a<a_*$ and $N=2$; or  (ii)  $a\le 0$ and $N=3$,
then  \eqref{8-18-1}--\eqref{8-28-3} has a ground state.  On the other hand,  \eqref{8-18-1}--\eqref{8-28-3} has no ground state if
 (i')  $a\ge a_*$ and $N=2$; or (ii') $a> 0$ and $N=3$.}

\medskip

In the last few years, lots of efforts have been made to the study of asymptotic behaviors of the minimizers of \eqref{1-7-11} as $a\nearrow a_*$ when $N=2$. See for example
\cite{GS,Guo,Guo1} and the references therein, where the main tools used are the energy comparison. The main results on the  asymptotic behaviors of the minimizer $u_a$
of \eqref{1-7-11} with $N=2$ as  $a\nearrow a_*$ are that $u_a$ concentrates at a minimum point $x_0$ of $V(x)$.  That is, $u_a \to 0$ uniformly in $\mathbb R^N \setminus B_\theta(x_0)$
for any $\theta>0$, while $\max_{x\in B_\theta(x_0)} u_a(x)\to +\infty$.  However, if $N=3$, as $a  \nearrow  0$,  the minimizer $u_a$ of \eqref{1-7-11} approaches to
a minimizer of $u_0$ of $I_0$. Therefore, it is not obvious that  \eqref{8-18-1}--\eqref{8-28-3} has solutions $u_a$ concentrating at some points if $N=3$.

The aim of this paper is to investigate the excited states for \eqref{8-18-1}--\eqref{8-28-3}, especially those which exhibit the concentration phenomena.
For this purpose, we need to consider \eqref{8-18-1} from different point of views  as follows.

We first consider the following problem without constraint:
\begin{equation}\label{1-23-5}
\begin{cases}
-\Delta w + (\lambda + V(x) ) w = w^3, & \text{in}\; \mathbb R^N;\vspace{2mm}\\
w\in H^1(\mathbb R^N),
\end{cases}
\end{equation}
where $\lambda>0$ is a large parameter.  It is well known that for large $\lambda>0$, we can construct various positive solutions concentrating at
some stable critical points of $V(x)$.  In particular, we can construct positive $k$-peak solutions for \eqref{1-23-5} in the sense that
\begin{equation}\label{2-23-5}
w_\lambda (x)=  \sqrt\lambda  \Bigl(\sum^k_{i=1} Q\bigl( \sqrt\lambda(x-x_{\lambda,i})\bigr)+\omega_\lambda (x)\Bigr),
\end{equation}
with  $\displaystyle\int_{\R^N}\big[\frac{1}{\lambda}|\nabla \omega_\lambda|^2+  \omega^2_\lambda\big]=o\big(\lambda^{-\frac{N}{2}}\big)$.
Let $
u_\lambda= \frac { w_\lambda} {\bigl(\int_{\mathbb R^N} w_\lambda^2\bigr)^{1/2}}$.
Then $\displaystyle\int_{\mathbb R^N} u_\lambda^2=1$, and
\begin{equation}\label{3-23-5}
\begin{cases}
-\Delta u_\lambda + (\lambda + V(x) ) u_\lambda = a_\lambda u_\lambda^3, & \text{in}\; \mathbb R^N;\vspace{2mm}\\
u_\lambda\in H^1(\mathbb R^N),
\end{cases}
\end{equation}
with
\[
a_\lambda= \int_{\mathbb R^N} w_\lambda^2 =k\lambda^{1-\frac N2} \Bigl( \int_{\mathbb R^N} Q^2 +o(1)\Bigr)
=k\lambda^{1-\frac N2} \Bigl(a_* +o(1)\Bigr).
\]
Note that $a_\lambda>0$, and as $\lambda \to +\infty$, $a_\lambda\to ka_*$ if $N=2$, while $a_\lambda\to 0$ if $N=3$.  Therefore, we obtain
a concentrated solution with $k$ peaks for \eqref{8-18-1}--\eqref{8-28-3} with $\mu=-\lambda$ and suitable $a_\lambda$.  Now the crucial
question is for any $a$ close to $ka_*$ if $N=2$, or for any $a>0$ small if $N=3$, whether we can choose a suitable large $\lambda_a>0$, such that
\eqref{8-18-1}--\eqref{8-28-3} holds with
\begin{equation}\label{10-23-5}
\mu=-\lambda_a,\quad u_a= \frac { w_{\lambda_a}} {\bigl(\displaystyle\int_{\mathbb R^N} w_{\lambda_a}^2\bigr)^{1/2}}.
\end{equation}

The above discussions show that the existence of concentrated solutions for \eqref{8-18-1}--\eqref{8-28-3} is closely related to the existence of peaked solutions
for the nonlinear Schr\"odinger equations \eqref{1-23-5}.
In this paper, we will mainly investigate   concentrated  solutions $u_a$  of  \eqref{8-18-1}--\eqref{8-28-3} in the  sense that
$$\max_{x\in B_\theta(b_{i})} u(x) \to +\infty,~\mbox{while}~u_a(x)\to 0~\mbox{uniformly in}~\mathbb R^N\setminus \bigcup_{i=1}^k  B_\theta(b_{i}), ~\mbox{for any}~\theta>0,$$
as $a \to a_0$, where $k$ is a positive integer and $b_{1},\cdots,b_k$ are some points in $\mathbb R^N$.
Here, we will study the following basic
issues concerning the concentration of the solutions for \eqref{8-18-1}--\eqref{8-28-3}:

 \smallskip

\noindent\emph{\textup{(I)} The possible values for $a_0$ and $\mu_a$, and the exact location of  the concentrated points, if $u_a$ concentrates.}

 \smallskip

\noindent\emph{\textup{(II)} The existence of the concentrated solutions.}

 \smallskip

\noindent\emph{\textup{(III)}  The local uniqueness of the concentrated solutions.}


 \smallskip

  Our first result of this paper is
the following.

\begin{Thm}\label{th1-24-5}
Suppose  $u_a$ is a   solution of \eqref{8-18-1}--\eqref{8-28-3} concentrated at some points
as $a\to a_0$.  Then it holds
\begin{equation*}
a_0\ge 0~\mbox{and}~\mu_a\to -\infty,~\mbox{as}~a\to a_0.
\end{equation*}
 Moreover,  if $N=2$, then   $a_0= k a_*$ for some integer $k>0$,  and $u_a$ satisfies
\begin{equation}\label{20-23-5}
u_a (x)=  \sqrt{\frac{-\mu_a}{a}} \Bigl(\sum_{i=1}^k  Q\bigl( \sqrt {-\mu_a}(x-x_{a, i})\bigr)+\omega_a(x)\Bigr),
\end{equation}
with $\displaystyle\int_{\R^2}\big[-\frac{1}{\mu_a}|\nabla \omega_a|^2+\omega_a^2\big]=o\big(-\frac{1}{\mu_a}\big)$ and  some points $x_{a,i}\in \mathbb R^2$,
$i=1,\cdots,k$, satisfying  $\sqrt {-\mu_a}| x_{a, j}-x_{a, i}|  \to \infty$ as $a\to a_0$, $ i\ne j$.

\end{Thm}

Throughout this paper, we call $u_a$  a $k$-peak solution of \eqref{8-18-1}--\eqref{8-28-3}
if $u_a$ satisfies \eqref{20-23-5}.
Although the  nonlinear Schr\"odinger equations \eqref{1-23-5} have been extensively studied in the last
 three decades (see for example \cite{Ambrosetti,Byeon,Cao4,Floer,Rabinowitz} and the references therein), not much is known for the exact location of the concentrated point, nor for the local uniqueness
of the solutions,  if the critical points of $V(x)$ is not isolated.
In the paper, we assume that $V(x)$ obtains its local minimum or local maximum at
 $\Gamma_i$ ($i=1,\cdots,k$) and $\Gamma_i$ is a closed $N-1$ dimensional hyper-surface satisfying
 $\Gamma_i \bigcap \Gamma_j=\emptyset$ for $i\neq j$.
  More precisely, we assume that the following conditions hold.

\medskip

\noindent\emph{\textup{\textbf{($V$).}} There exist $\delta>0$  and some $C^2$ compact  hypersurfaces $\Gamma_i~(i=1,\cdots,k)$  without boundary,  satisfying
$$
V(x)=V_i, \;\;  \frac{\partial V(x)}{\partial \nu_i}=0,\;\; \frac{\partial^2 V(x)}{\partial \nu_i^2}\ne 0, ~\mbox{for any}~x \in \Gamma_i~~\mbox{and}~i=1,\cdots,k,$$
where $V_i\in \R$, $\nu_i$ is the unit outward normal of  $\Gamma_i$  at $x\in \Gamma_i$.   Moreover, $V\in C^4\big(\bigcup^k_{i=1}W_{\delta,i}\big)$ and $V(x)=O\big(e^{\alpha|x|})$ for some $\alpha \in (0, 2)$.
Here we denote  $W_{\delta,i}:=\{x\in \R^N, dist(x,\Gamma_i)<\delta\}$.
}

\smallskip

Note that condition $(V)$ implies that $V(x)$ obtains its local minimum or local maximum on the hypersurface $\Gamma_i$ for $i=1,\cdots,k$. It is also easy
 to see that if $\delta>0$ is small,  the  set $\Gamma_{t,i}=\bigl\{  x:  V(x)= t\bigr\}\bigcap  W_{\delta,i}$ consists of two   compact  hypersurfaces
in $\mathbb R^N$  without boundary for $t\in [V_i, V_i+\theta]$
( or $ t\in [V_i-\theta, V_i]$) provided $\theta>0$ is small. Moreover,   the outward unit normal vector $\nu_{t,i}(x)$ and   the $j$-th principal tangential unit vector $\tau_{t,i,j}(x)$  ($j=1, \cdots, N-1$) of  $\Gamma_{t,i}$  at
$x$ are Lip-continuous   in $W_{\delta,i}$.

Using the local Pohozaev identities, we can easily prove that a $k$-peak solution of \eqref{8-18-1}--\eqref{8-28-3} must concentrate at some critical points of $V(x)$, and we can also refer to
\cite{Grossi}.
If $V(x)$ satisfies $(V)$  and the concentrated points belong to $\Gamma:=\bigcup^k_{i=1}\Gamma_i$,  it is natural to ask where the concentrated points locate on $\Gamma$.
And  the following result gives the further answer of this question.
\begin{Thm}\label{nth1.1}
Under the condition \textup{($V$)}, if $u_a$ is a $k$-peak solution of \eqref{8-18-1}--\eqref{8-28-3}, concentrating at $\{b_1,\cdots,b_k\}$ with $b_i\in \Gamma$  and
$b_i\ne b_j$ if $i\ne j$,
as $a\to ka_*$ and  $N=2$,  or $a\searrow 0$ and $N=3$, then
\begin{equation}\label{1.6}
(D_{\tau_{i,j}} \Delta  V)(b_i)=0,~\mbox{with}~
 i=1,\cdots,k~\mbox{and}~j=1,\cdots,N-1.
\end{equation}
where $\tau_{i,j}$ is the $j$-th  principal tangential unit vector of $\Gamma$ at $b_i$.
\end{Thm}

It is proved in \cite{GS} that
 if  $V(x)$ has finite minimal points with the same minimal value, then  the the minimizers of \eqref{1-7-11} concentrate at  the ``flattest"  minimal points of $V(x)$ along a subsequence $a_l$ which approaches  $ a_*$ from left as $l\to \infty$.  On the other hand, for $V(x)=(|x|-1)^2$, in  \cite{Guo}, it is proved that
   the minimizers of \eqref{1-7-11} concentrate at some point in $\{x: |x|=1\}$, while for
   $$V(x)=\Big(\big(\frac{x_1^2}{a^2}+\frac{x_2^2}{b^2}\big)^\frac12-1\Big)^2,~\mbox{with}~a>b>0,$$
    it is  proved in \cite{GZ} that
     the minimizers of \eqref{1-7-11} concentrate at either  $(-a,0)$, or $(a,0)$ up to a subsequence.  Note that
   in all those cases, the concentrated point is a minimum point of the function $\Delta V$ on the relevant set.

Theorem~\ref{nth1.1} shows that not every $\{b_1,\cdots,b_k\}$ with $b_j\in\Gamma$ can generate a $k$-peak solution for \eqref{8-18-1}--\eqref{8-28-3}.
 To study the converse of Theorem~\ref{nth1.1}, we need the following non-degenerate condition on the critical point  of $V(x)$.
   We say that  $x_0 \in \Gamma_i$  is non-degenerate on $\Gamma_i$ if the following condition holds:
\begin{equation*}
\frac{\partial^2V(x_0)}{\partial \nu_i ^2}\neq 0~~\mbox{and}~~
det \Big(\Big(\frac{\partial^2 \Delta V(x_0)}{\partial \tau_{i,l}\partial \tau_{i,j}}\Big)_{1\leq l,j\leq N-1}\Big)\neq 0.
\end{equation*}

\begin{Thm}\label{nth1.2}
 Assume that \textup{($V$)}   holds.   If   $b_i\in \Gamma$ is  non-degenerate critical point of $V(x)$ on $\Gamma$ for $i=1,\cdots,k$ satisfying \eqref{1.6} and
 $b_i\ne b_j$ if $i\ne j$,
  then there exists a small constant $\theta>0$, such that \eqref{8-18-1}--\eqref{8-28-3}
  has  a $k$-peak solution $u_a$ concentrating at $b_1,\cdots,b_k$ as $a\to ka_*$ if $N=2$, provided

 \textup(i). $a\in [ka_*-\theta, ka_*)$, if  $ \displaystyle\sum^k_{i=1}\Delta V(b_i)>0$,~~~~
 \textup(ii). $a\in (ka_*, ka_*+\theta]$, if  $\displaystyle\sum^k_{i=1}\Delta  V(b_i)<0$,\\
or as $a\in (0, \theta]$ if $N=3$.
\end{Thm}

The existence result in Theorem~\ref{nth1.2} is new even when $k=1$ because it reveals that  \eqref{8-18-1}--\eqref{8-28-3} still has single peak solutions for $a>a_*$ if
$V(x)$ has a local maximum point or local maximum set. Let us point out that if $V(x)$ does not achieve its global minimum on $\Gamma$, any
 solution concentrating at a point on $\Gamma$ is not a ground state. Also, if $k>1$, then the solutions in Theorem \ref{nth1.2} are not ground states. So Theorem \ref{nth1.2} gives us the  existence of excited states for BEC problem as $a\rightarrow ka_*$ if $N=2$,  or $a\searrow 0$ if $N=3$. To our best knowledge, this is the first result concerning the existence of  excited states for  \eqref{8-18-1}--\eqref{8-28-3}.
Furthermore, we can prove  that for any integer $k>0$ and some $a$ near $k a^*$ in $N=2$,  or $a$ near $0$ in $N=3$, \eqref{8-18-1}--\eqref{8-28-3}  has an
 excited state solution which has $k$-peaks concentrated at one point (see Theorem \ref{nnth1.2} in Section 4).

\smallskip

Another main result of this paper is the following local uniqueness result.

\begin{Thm}\label{th1.4}

Suppose   \textup{($V$)}   holds.  Let $u_a^{(1)}(x)$ and $u_a^{(2)}(x)$ be two  $k$-peak solutions of \eqref{8-18-1}--\eqref{8-28-3} concentrating  at $b_1,\cdots,b_k$ with $b_i\in \Gamma$,  and $b_i\ne b_j$ if  $i\ne j$.  If $ b_i$ is non-degenerate, $i=1,\cdots, k$, 
 $
\displaystyle\sum^k_{i=1}\Delta V(b_i)\neq 0$, in $N=2$,
and
\[
\Big(\frac{\partial^2 \Delta V(b_i)}{\partial \tau_{i,l} \partial\tau_{i,j}}\Big)_{1\leq l,j\leq N-1}+\frac{\partial  \Delta V(b_i)}{\partial \nu_i}  diag \big(\kappa_{i,1}, \cdots, \kappa_{i,N-1}\big),~\mbox{for}~i=1,\cdots,k
\]
is non-singular, where $\kappa_{i,j}$ is the $j$-th principal curvature of $\Gamma$ at $b_i$ for $j=1,\cdots,N-1$,
then there exists a small positive number $\theta$,  such that $u_a^{(1)}(x)\equiv u_a^{(2)}(x)$ for all $a$ with $0<|a-ka_*|\le \theta $ if $N=2$, or $0< a \le \theta $ if $N=3$.
\end{Thm}

As far as we know, local uniqueness results for peak (or bubbling) solutions are available only for the case where the solutions blow up
at $x_0$, which is an isolated critical point of the potential $V(x)$.   If $x_0$ is a non-degenerate critical point of $V(x)$, that is,  $ (D^2 V)$ is non-singular at $x_0$, one
can prove the local uniqueness of the peak solution concentrating at $x_0$ either by counting the local degree of the
corresponding reduced finite dimensional  problem as in  \cite{Cao3,CNY,G}, or by using Pohozaev type identities as in  \cite{Cao1,Deng,Grossi,GPY,GLW}.
One of the advantage in using the Pohazaev identities to prove the local uniqueness is that it can deal with the degenerate case. See  \cite{Cao1,Deng,GPY}.
Let us pointing out that in  \cite{Cao1,Deng,GPY}, though the critical point $x_0$ is degenerate, the rate of degeneracy along each direction is the same.
On the other hand,  an example given in \cite{Grossi} shows that  local uniqueness may not be true at a degenerate critical point $x_0$ of $V(x)$.
Thus, it is a very subtle problem to discuss the local uniqueness of peak solutions concentrating at a degenerate critical point.
Under the condition ($V$), the function $V(x)$ is non-degenerate along the normal direction $\nu_i$ of $\Gamma_i$. But along each tangential direction of $\Gamma_i$,
$V(x)$ is degenerate. Such non-uniform degeneracy makes the estimates more sophisticated.

Here we point out that the existence and local uniqueness of excited states to \eqref{8-18-1}--\eqref{8-28-3}
are also true for the following type of potential $V(x)$:
\begin{equation}\label{lll111}
V(x)=\prod^m_{i=1}|x-x_i|^{p_i},~\mbox{for all}~x\in \R^N~\mbox{with}~p_i>0 ~\mbox{and}~N=2,3.
\end{equation}
Also our arguments in this paper show that
it is much more effective to use the Pohozaev identities to study the asymptotic behaviors for all kinds of concentrated solutions, not just
for the minimizer. For example, using various  Pohozaev identities, we can easily derive the relation between $a$ and the Lagrange multiplier $\mu_a$ (see the proof of Proposition~\ref{p1-711}).

This paper is organized as follows.  In section~2, we will prove Theorem~\ref{th1-24-5}, while in section~3, we estimate the Lagrange multiplier $\mu_a$
in terms of $a$.  The results for the location of the peaks and  for the existence of peak solutions are proved in  section~4, and the local uniqueness of
peak solutions are investigated in section~5.

 For simplicity in using the notations, in this paper, we always assume that $b_j\in \Gamma_j$, $j=1,\cdots, k$. The results for other cases can
 be proved without any changes.

\section{A non-existence result and the Proof of Theorem~\ref{th1-24-5}}
First, we study the following problem:
\begin{equation}\label{1-26-10}
-\Delta u =   V_1(x)u, \; u>0, \quad \text{in}\; \mathbb R^N,
\end{equation}
where the function $V_1(x)$ satisfies $V_1> 1$  in $B_R(0)\setminus B_{t}(0)$ for some fixed $t>0$ and large $R>0$.

\begin{Prop}\label{th1-26-10}

Problem~\eqref{1-26-10} has no solution.

\end{Prop}

\begin{proof}

Suppose that \eqref{1-26-10} has a solution $u$.
Consider the following problem
\begin{equation}\label{2-26-10}
-\Delta v =  v.
\end{equation}
By a standard comparison argument,  \eqref{2-26-10} has a radial solution $v(r)$, which has infinitely many
zeros points $0< r_1<r_2<\cdots< r_k<\cdots$.  Denote   $\Omega_k = B_{r_{k+1}}(0)\setminus B_{r_{k}}(0)$.  Let  $k_0>0$ be such that $ t< r_{k_0}$.  We now assume that  $R>> r_{k_0 }$.  We take $k \ge k_0$,  such that
  $v>0$ in $\Omega_{k}$,  then, we have
\begin{equation}\label{3-26-10}
\int_{\Omega_k}  \Bigl(-v \Delta u  + u\Delta v\Bigr) = \int_{\Omega_k} \Bigl( V_1(x) u v -  vu\Bigr)>0.
\end{equation}
Noting that $v(r_k)= v(r_{k+1})=0$, we obtain from \eqref{3-26-10} that
\begin{equation}\label{4-26-10}
\int_{\partial \Omega_k} u \frac{\partial v}{\partial \nu} >0,
\end{equation}
where $\nu$ is the outward unit normal of $\partial \Omega_k$.  But on $\partial \Omega_k$, $\frac{\partial v}{\partial \nu}<0$. Thus, we obtain a contradiction from \eqref{4-26-10}.
\end{proof}

\begin{proof}[\textbf{Proof of Theorem~\ref{th1-24-5}}]

First, we prove that $\mu_a\to -\infty$. We argue by contradiction.
Suppose that  $|\mu_a|\leq M$.  Since  $\int_{\mathbb R^N} u_a^2=1$, Moser iteration implies that  $u_a$ is uniformly bounded. That is, $u_a$  does not blow up.

 Suppose that  $\mu_a\rightarrow +\infty$.  We  let $V_1(x)= \mu_a - V(x) + a u^2_a$.  Since $u_a$ concentrates at some points,
we may assume that $a u_a^2\ge -1$ in $\mathbb R^N\setminus B_t(0)$ for some $t>0$.  Therefore, for any fixed $R>0$, we always have

\[
V_1(x)= \mu_a - V(x) + a u^2_a>1,\quad x\in B_R(0)\setminus B_t(0).
\]
By  Proposition~\ref{th1-26-10}, we obtain a contradiction.  Therefore, we have proved that $\mu_a\to -\infty$.  Let  $\lambda_a=-\mu_a$.
Let $x_a$ be the maximum point of $u_a$. From the equation \eqref{8-18-1}, we find

\[
 a u_a^3(x_a) \ge \big(\lambda_a + V(x_a)\big) u_a(x_a)>0.
 \]
 This gives $a>0$.

Suppose now $N=2$. Let $\bar u_a (x)= \frac1{\sqrt{\lambda_a}} u_a\bigl( \frac x{\sqrt{\lambda_a}}\bigr)$.  Then

\begin{equation}\label{30-27-5}
-\Delta \bar u_a + \Bigl( 1  + \frac1{ {\lambda_a}} V\bigl( \frac x{\sqrt{\lambda_a}}\bigr)\Bigr)\bar u_a  = a \bar u_a^3,\quad \text{in}\; \mathbb R^2,
\end{equation}
and

\begin{equation}\label{31-27-5}
\int_{\mathbb R^2} \bar u_a^2=1.
\end{equation}

 From \eqref{30-27-5} and \eqref{31-27-5}, using Moser iteration, we can prove that $|u_a|\le C$ for some constant independent of $a$.  Let $\bar x_a$ be a maximum point of $\bar
 u_a$.  Then

\[
a u_a^3(\bar x_a)\ge \Bigl( 1  + \frac1{ {\lambda_a}} V\bigl( \frac {\bar x_a}{\sqrt{\lambda_a}}\bigr)\Bigr)\bar u_a(\bar x_a),
\]
 which gives $
 a\ge   u_a^{-2}(\bar x_a)  \ge c_0>0$.
Using the standard blow-up argument, in view of \eqref{31-27-5},  we can prove that there exists an integer $k>0$, such that

 \begin{equation}\label{32-27-5}
 \bar u_a =\sum_{i=1}^k  Q_{a_0} (x-\bar x_{a, i})+\bar{\omega}_a(x),
 \end{equation}
 for some $\bar x_{a, i} \in \mathbb R^2$ with
 $$|\bar x_{a, i}-\bar x_{a, j}|\to +\infty,~\mbox{if}~i\ne j,~\displaystyle\int_{\R^2}\big[|\nabla\bar{\omega}_a|^2
 +\bar{\omega}^2_a\big]=o(1),$$
and $Q_{a_0}$ is the unique positive solution of

 \[
 -\Delta u + u = a_0 u^3,\;\; u\in H^1(\mathbb R^2),\;\; u(0)=\max_{x\in \mathbb R^2} u(x).
 \]
 Noting that  $Q_{a_0}=\frac1{\sqrt{a_0}} Q$, we obtain from \eqref{31-27-5}  and \eqref{32-27-5}  that  $a_0= k a_*$.

\end{proof}

\section{Some estimates for general potentials}

In this section,  we shall estimate $\mu_a$  in terms of $a$.

Let $\varepsilon =\frac{1}{ \sqrt{-\mu_a}}$ and $u(x)\mapsto \sqrt{\frac{-\mu_a}{a}} u(x)$,
then \eqref{8-18-1} can be changed to the following problem:
\begin{equation}\label{30-7-11}
-\varepsilon^2\Delta u+ \bigl( 1+ \varepsilon^2 V(x)\bigr)u= u^3,~
u\in H^1(\R^N).
\end{equation}
For any $a\in \R^+$, we define
 $\|u\|_a:=\displaystyle\int_{\mathbb R^N} \bigl( \varepsilon^2 |\nabla u|^2 + u^2\bigr)$.

 From \eqref{20-23-5}, we find  that a $k$-peak solution of \eqref{30-7-11} has  the following form
\begin{equation*}
\tilde{u}_a(x)= \sum_{i=1}^k Q_{\varepsilon,x_{a,i}}(x) +v_{a}(x),~\mbox{with}~~\|v_{a}\|_a=o(\varepsilon^{\frac N2}),
\end{equation*}
where $Q_{\varepsilon,x_{a,i}}(x):=Q\Big(\frac{\sqrt{1+\varepsilon^2V_i}(x-x_{a, i})}\varepsilon \Big)$.
Then, it holds
 \begin{equation}\label{06-07-3}
    -\varepsilon^2\Delta v_{a}+\Big(1+\varepsilon^2V(x)-
    3\sum_{i=1}^kQ^2_{\varepsilon,x_{a,i}}(x)\Big)
    v_{a}
    =N\big(v_{a}\big)+l_{a}(x),
 \end{equation}
where
\begin{equation}\label{06-07-1}
      N_a\big(v_{a}\big)= \Big(\sum_{i=1}^k Q_{\varepsilon,x_{a,i}}(x)
    +v_{a}\Big)^{3}-
    \Big(\sum_{i=1}^k Q_{\varepsilon,x_{a,i}}(x)\Big)^3
    -3\sum_{i=1}^k  Q^2_{\varepsilon,x_{a,i}}(x)v_{a},
\end{equation}
and
$$ l_{a}(x)=-\varepsilon^2\sum_{i=1}^k\big(V(x)-V_i\big)Q_{\varepsilon,x_{a,i}}(x)
+ \Big(\sum_{i=1}^k Q_{\varepsilon,x_{a,i}}(x)\Big)^3-
    \sum_{i=1}^k Q^3_{\varepsilon,x_{a,i}}(x).$$
 We can move $x_{a,i}$ a bit(still denoted by $x_{a,i}$), so that the error term $v_a\in \displaystyle\bigcap^k_{i=1}E_{a,x_{a,i}}$, where
\begin{equation}\label{06-08-1}
E_{a,x_{a,i}} :=\left \{ u(x)\in H^1(\R^N):
\Big\langle u,\frac{\partial Q_{\varepsilon,x_{a,i}}(x)}{\partial{x_j}}
\Big\rangle_{a}=0, ~j=1,\cdots,N
 \right\}.
 \end{equation}

 Let  $L_{a}$ be the bounded linear operator from  $H^1(\mathbb R^N)$ to itself, defined by
\begin{equation}\label{06-07-2}
\bigl\langle L_{a} u, v\bigr\rangle_a=\int_{\mathbb R^N}  \Bigl(  \varepsilon^2\nabla  u\nabla v
 +\big(1+\varepsilon^2V(x)\big)u v-3\sum_{i=1}^k   Q^2_{\varepsilon,x_{a,i}}(x)
u  v  \Bigr).
\end{equation}
Then, it is standard to prove the following lemma.
\begin{Lem}\label{lem-A.1}
There exist  constants $\rho>0$ and small $\theta>0$,  such that for all $a$ with $0<|a-ka_*|\le \theta $ in $N=2$ or $0< a \le \theta $ in $N=3$, it holds
 \begin{equation}\label{B.2}
   \|L_{a}u\|_{a}\geq \rho \|u\|_{a},~\mbox{for all}~u\in \bigcap^k_{i=1}E_{a,x_{a,i}}.
 \end{equation}

\end{Lem}
\begin{Lem}\label{lem-5-05-1}
A $k$-peak solution $\tilde{u}_a$    for \eqref{30-7-11}  concentrating at $b_1,\cdots,b_k$ has  the following form
\begin{equation}\label{aaa8-20-1}
\tilde{u}_a(x)= \sum_{i=1}^k   Q_{\varepsilon,x_{a,i}}(x)+v_{a}(x),
\end{equation}
with  $v_a \in \displaystyle\bigcap^k_{i=1}E_{a,x_{a,i}}$  and
\begin{equation}\label{lt1}
\|v_{a}\|_a=O\Big(
\big|\sum^k_{i=1}\big(V(x_{a,i})-V_i\big)\big| \varepsilon^{{\frac{N}{2}+2} } +\big| \sum^k_{i=1} \nabla V(x_{a,i})
   \big|\varepsilon^{\frac{N}{2}+3}+ \varepsilon^{\frac{N}{2}+4} \Bigr).
\end{equation}
\end{Lem}
\begin{proof}

 By standard calculations, we find
 \begin{equation}\label{10-31-7}
   \|l_{a}\|_{a}=O\Big(\big|\sum^k_{i=1}\big(V(x_{a,i})-V_i\big)\big|\varepsilon^{{\frac{N}{2}+2} } +\big| \sum^k_{i=1} \nabla V(x_{a,i})
   \big|\varepsilon^{\frac{N}{2}+3}+ \varepsilon^{\frac{N}{2}+4} \Bigr),
 \end{equation}
 and
\begin{equation}\label{aaab10-31-7}
    \|N\big(v_{a}\big)\|_{a}=o(1)\|v_{a}\|_{a}.
\end{equation}
Then from \eqref{06-07-3}, \eqref{B.2}, \eqref{10-31-7} and \eqref{aaab10-31-7},  we get  \eqref{aaa8-20-1} and \eqref{lt1}.

\end{proof}

Let  $\tilde v_{a}(x) =v_{a}(\varepsilon x + x_{a,i})$. Then, $\tilde v_{a}$ satisfies
$\|\tilde v_{a}\|_a=O\big(\varepsilon^2\big)$. Using the Moser iteration, we can prove $
\|\tilde v_{a}\|_{L^\infty(\mathbb R^N)}=o(1)$.
 From this  and the comparison theorem, we can prove the following estimates for $\tilde u_a(x)$ away from the concentrated points $b_1,\cdots,b_k$.

\begin{Prop}\label{prop1-1}
Suppose that $\tilde u_a(x)$ is a $k$-peak solution  of  \eqref{30-7-11} concentrated at $b_1,\cdots,b_k$.  Then
for any fixed $R\gg 1$, there exist some $\theta>0$ and $C>0$, such that
\begin{equation}\label{2--5}
|\tilde u_a(x)|+|\nabla \tilde u_a(x)|\leq C\sum^k_{i=1}e^{-\theta |x-x_{a,i}|/\varepsilon},~\mbox{for}~
x\in \R^N\backslash \bigcup^k_{i=1}B_{R \varepsilon}(x_{a,i}).
\end{equation}
\end{Prop}

\begin{Lem}
It holds
\begin{equation}\label{06-13-1}
(4-N)\displaystyle\int_{\R^N} Q^4=4\int_{\R^N}Q^2,~\mbox{for}~N=2,3.
\end{equation}
\end{Lem}
\begin{proof}
It is direct by the following  Pohozaev identities:
\begin{equation}\label{a5-21-1}
 \int_{\R^N}|\nabla Q|^2+\int_{\R^N} Q^2
=\int_{\R^N} Q^4,~~
 (N-2)\int_{\R^N}|\nabla Q|^2+N\int_{\R^N} Q^2
=\frac{N}{2}\int_{\R^N} Q^4.
\end{equation}
\end{proof}
\begin{Prop}\label{p1-711}
Let $N=2$ and $a\rightarrow ka_*$,
it holds
\begin{equation}\label{8-29-1}
(ka_*-a)  \mu^2_a= \frac{1}{2} \sum^k_{i=1}\Delta V(b_{i}) \displaystyle\int_{\R^2}|x|^2Q^2(x) +o\big(1\big).
\end{equation}
\end{Prop}
\begin{proof}Let $u_{a}$ be a  solution of \eqref{8-18-1}--\eqref{8-28-3}, multiplying
$\langle x-x_{a,i},\nabla u_{a}\rangle$ on both sides of \eqref{8-18-1} and integrating on $B_d(x_{a,i})$, we find
\begin{equation}\label{8-19-3}
 \int_{B_d(x_{a,i})}\big[\big(2V(x)+\langle \nabla V(x), x-x_{a,i}\rangle \big) u_a^2
-2\mu_a  u_a^2-a u_a^4\big]
= \int_{\partial B_d(x_{a,i})} W(x)d\sigma,
\end{equation}
where
$$W(x):=-2\frac{\partial u_a}{\partial\nu}\langle x-x_{a,i},\nabla u_{a}\rangle
+\langle x-x_{a,i},\nu\rangle \big[|\nabla u_{a}|^2+V(x)u^2_{a}-\mu_au^2_{a}-\frac{a}{2}u^4_{a}\big].$$
Also, from \eqref{aaa8-20-1}, \eqref{lt1} and \eqref{20-23-5}, we can write $u_a(x)$ as follows:
\begin{equation}\label{bac8-20-1}
u_a(x)=\sum^k_{i=1}
\sqrt{\frac{-\mu_a+V_i}{a}} \Big(\bar Q_{a,x_{a,i}}(x)+v_{a}(x)\Big).
\end{equation}
where $\bar Q_{a,x_{a,i}}(x):=Q\Big(\sqrt{-\mu_a+V_i}(x-x_{a,i})\Big)$
and $\|v_{a}\|_a=O\Big(\frac{1}{(\sqrt{-\mu_a})^3}\Big)$.

Hence, using \eqref{bac8-20-1},
 we get
\begin{equation}\label{bbaaa8-19-3}
\begin{split}
 \int_{B_d(x_{a,i})}&\big(2V(x)+\langle \nabla V(x), x-x_{a,i}\rangle \big) u_a^2
 \\=&\frac{-\mu_a+V_i}{a}\displaystyle\int_{B_d(x_{a,i})}\Big(\langle \nabla V(x), x-x_{a,i}\rangle+2
 \big(V(x)-V_i\big)\Big)
 \bar Q^2_{a,x_{a,i}}(x) \\&
 +2V_i\int_{B_d(x_{a,i})} u_a^2
 +O\Big(-\frac{1}{\mu_a^3}\Big)\\=&
 \frac{1}{a(-\mu_a+V_i)} \Delta V(x_{a,i}) \int_{\R^2}|x|^2Q^2(x)+2V_i\frac{a^*}{a} +o\Big(-\frac{1}{\mu_a}\Big).
\end{split}
\end{equation}
Also, we have
\begin{equation}\label{bcaaa8-19-3}
\Delta V(x_{a,i})= \Delta V(b_i)+O(|x_{a,i}-b_i|)=\Delta V(b_i)+o(1).
\end{equation}
Then from \eqref{bbaaa8-19-3} and \eqref{bcaaa8-19-3}, we get
\begin{equation}\label{abaaa8-19-3}
\begin{split}
 \int_{B_d(x_{a,i})} &\big(2V(x)+\langle \nabla V(x), x-x_{a,i}\rangle \big) u_a^2
\\=&
\frac{1}{a(-\mu_a+V_i)} \Delta V(b_i) \int_{\R^2}|x|^2Q^2(x)+2V_i\int_{B_d(x_{a,i})} u_a^2+o\big(-\frac{1}{\mu_a}\big).
\end{split}
\end{equation}
So summing \eqref{8-19-3} from $i=1$ to $i=k$ and using \eqref{abaaa8-19-3}, we find
\begin{equation}\label{6-18-21}
\begin{split}
2\mu_a +a \int_{\R^2} u_a^4=
\sum^k_{i=1} \frac{1}{a(-\mu_a+V_i)}\Delta V(b_i) \int_{\R^2}|x|^2Q^2(x) +\frac{2a^*}{a} \sum^k_{i=1}V_i
+o\big(-\frac{1}{\mu_a}\big).
\end{split}
\end{equation}
On the other hand,
we can obtain
\begin{equation}\label{8-27-21}
\begin{split}
\int_{\R^2}u^4_a
=&
\sum^k_{i=1}\frac{(-\mu_a+V_i)^2}{a^2}\Big[\int_{\R^2}  \bar Q^4_{a,x_{a,i}}(x)
+4 \int_{\R^2} \bar Q^3_{a,x_{a,i}}(x) v_a\\&
+O\Big( \int_{\R^2}  \bar Q^2_{a,x_{a,i}}(x)v^2_{a}
+O\Big(\int_{\R^2}v_{a}^4\Big) \Big]+o\Big(-\frac{1}{\mu_a}\Big)\\=&
\frac{2 a_* }{a^2}\sum^k_{i=1}(-\mu_a+V_i) +O\Big(\mu_a^2\|v_a\|^2_{a} \Big)
=
\frac{2 a_* }{a^2}\sum^k_{i=1}(-\mu_a+V_i) +o\Big(-\frac{1}{\mu_a}\Big).
\end{split}
\end{equation}
Here we use the fact that $\displaystyle\int_{\R^2} \bar Q^3_{a,x_{a,i}}(x) v_a=0$ for $i=1,\cdots,k$ by orthogonality.

Then  \eqref{6-18-21} and \eqref{8-27-21} imply
\begin{equation*}
\mu_a(a-ka_*)= \frac{1}{2}\sum^k_{i=1} \frac{1}{-\mu_a+V_i}\Delta V(b_i) \displaystyle\int_{\R^2}|x|^2Q^2(x)
+o\big(-\frac{1}{\mu_a}\big),
\end{equation*}
which gives  \eqref{8-29-1}.
\end{proof}

 \begin{Prop}\label{p1-7111}
Let $N=3$ and $a\searrow 0$,
it holds
\begin{equation}\label{ab8-29-1}
a \sqrt{-\mu_a}= ka_*
+\sum^k_{i=1} \frac{V_i}{2\mu_a}+O\Big(\big|\sum^k_{i=1}V(x_{a,i})\big|\frac{1}{
 (-\mu_a)}
  +\big|\sum^k_{i=1} \nabla V(x_{a,i})
   \big|\frac{1}{(\sqrt{-\mu_a})^3 }+ \frac{1}{\mu_a^2 } \Bigr).
\end{equation}
\end{Prop}
\begin{proof}
From \eqref{8-18-1} and \eqref{8-28-3}, we have
\begin{equation}
\begin{split}
1=&\int_{\R^3}\big({u}_a(x)\big)^2=
\int_{\R^3}\Big(\sum_{i=1}^k
\sqrt{\frac{-\mu_a+V_i}{a}} \Big(\bar Q_{\varepsilon,x_{a,i}}(x)+v_{a}(x)\Big) ^2,
\end{split}\end{equation}
which gives
\begin{equation*}
\begin{split}
a=
\sum^k_{i=1}\frac{a_*}{\sqrt{-\mu_a+V_i}} +O\Big(\big|\sum^k_{i=1}(V(x_{a,i})-V_i)\big|
 \frac{1}{(\sqrt{-\mu_a})^3 }+\big| \sum^k_{i=1} \nabla V(x_{a,i})
   \big|\frac{1}{\mu_a^2}+\frac{1}{(\sqrt{-\mu_a})^5}\Bigr).
\end{split}\end{equation*}
Then we can find
\eqref{ab8-29-1}.
\end{proof}

\section{ Locating the peaks and the Existence of peak solutions}

 First, we locate the peaks for a $k$-peak solution. Let $\tilde{u}_a$ be a $k$-peak solution of \eqref{30-7-11},
then for any small fixed $d>0$, we find
\begin{equation}\label{30-31-7}\begin{split}
\varepsilon^2&\int_{B_{d}(x_{a,i})} \frac{\partial V(x)}{\partial x_j}\big(\tilde{u}_a\big)^2  \\=&-2\varepsilon^2\int_{\partial B_{d}(x_{a,i})}\frac{\partial \tilde{u}_a)}{\partial \bar \nu}\frac{\partial \tilde{u}_a}{\partial x_j} +\varepsilon^2\int_{\partial B_{d}(x_{a,i})}|\nabla \tilde{u}_a|^2\bar\nu_j(x) \\&
 +\int_{\partial B_{d}(x_{a,i})}
\big(1+\varepsilon^2V(x)\big)\big(\tilde{u}_a\big)^2\bar\nu_j(x)\mathrm{d}
\sigma-\frac{1}{2}\int_{\partial B_{d}(x_{a,i})}|\tilde{u}_a|^4\bar\nu_j(x)\\
=&
O(e^{-\frac{\gamma}{\varepsilon}}), ~
\mbox{with some}~\gamma>0,
\end{split}\end{equation}
where $i=1,\cdots,k$, $j=1,\cdots,N$ and $\bar \nu(x)=\big(\bar\nu_{1}(x),\cdots,\bar\nu_N(x)\big)$ is the outward unit normal of $\partial B_{d}(x_{a,i})$.
And then \eqref{30-31-7} implies the first necessary condition for the concentrated points $b_i$:
$$\nabla V(b_i)=0,~\mbox{for}~i=1,\cdots,k.$$

\begin{proof}[\textbf{Proof of Theorem \ref{nth1.1}:}]

Since $x_{a,i}\to b_i\in \Gamma_{i}$, we find that there is a $t_a\in [V_{i}, V_{i}+\theta]$ if $\Gamma_{i}$
is a local minimum set of $V(x)$, or $t_a\in [V_{i}-\theta, V_{i}]$ if $\Gamma_{i}$
is a local maximum set of $V(x)$, such that $x_{a,i}\in \Gamma_{t_a,i}$.
Let $\tau_{a,i}$ be the unit tangential vector of $\Gamma_{t_a,i}$ at $x_{a,i}$.   Then
$$
G(x_{a,i}) =0,~\mbox{where}~G(x)=  \bigl\langle \nabla V(x), \tau_{a,i}\bigr\rangle.$$
Also we have the following expansion:
\begin{equation*}
\begin{split}
 G(x)=& \langle\nabla G( x_{a,i}), x-x_{a,i}\rangle+
\frac{1}{2}\big\langle \langle \nabla^2 G( x_{a,i}), x- x_{a,i}\rangle,x- x_{a,i}\big\rangle
+o\big(|x- x_{a,i}|^2\big),~\mbox{for}~x\in B_{d}(x_{a,i}).
\end{split}\end{equation*}  Then it follows from  \eqref{30-31-7} and the above expansion that
\begin{equation}\label{luo-6}
\begin{split}
\int_{B_{d}(x_{a,i})}& G(x) Q^2_{\varepsilon,x_{a,i}}(x)\\=&
-2 \int_{B_{d}(x_{a,i})} G(x)Q_{\varepsilon,x_{a,i}}(x) v_{a}-\int_{B_{d}(x_{a,i})} G(x) v^2_{a}+O\big(e^{-\frac{\gamma}{\varepsilon}}\big)\\=&
O\Big( [\varepsilon^{\frac{N}{2}+1} |\nabla G(x_{ a,i})| +\varepsilon^{\frac{N}{2}+2}]\|v_{a}\|_a+\varepsilon|\nabla G(x_{ a,i})|\cdot
\|v_{a}\|^2_a\Big)+O\big(e^{-\frac{\gamma}{\varepsilon}}\big)\\
=&
O\big( \varepsilon^{N+3}\big).
\end{split}
\end{equation}
Also,  in view of $ G(x_{a,i})=0$, it is easy to show
\begin{equation}\label{06-09-1}
\int_{\mathbb R^N} G(x)Q^2_{\varepsilon,x_{a,i}}(x)=
\frac{1}{2N} \varepsilon^{N+2} \Delta G(x_{a,i}) \displaystyle\int_{\R^N}|x|^2Q^2
+O\bigl (\varepsilon^{N+4}\bigr).
\end{equation}
Then \eqref{luo-6} and \eqref{06-09-1} give  $(\Delta G) (x_{a,i})=O\big(\varepsilon\big)$.
Thus by the condition $(V)$,  we obtain \eqref{1.6}.
\end{proof}

Now, we consider the  existence of peak solutions for  \eqref{1-23-5} with $\lambda>0$ a large parameter.
Let $\eta=\frac{1}{\sqrt{\lambda}}$ and $w(x)\mapsto \sqrt{\lambda} w(x)  $, then \eqref{1-23-5}
 can be changed to following problem:
\begin{equation}\label{20-7-11}
-\eta^2\Delta w+ \bigl( 1+ \eta^2 V(x)\bigr)w=w^3,~
w\in H^1(\R^N).
\end{equation}

In the following,  we denote  $\langle u, v\rangle_\eta = \int_{\mathbb R^N} \bigl( \eta^2 \nabla u\nabla v   +  uv\bigr)$ and $\|u\|_\eta=\langle u, u\rangle_\eta^{\frac12}
$.
Now for $\eta>0$ small, we construct a $k$-peak solution $u_\eta$ of \eqref{20-7-11}  concentrating at
$b_1,\cdots,b_k$.
Here  we can prove the following result in a standard way. 

\begin{Prop}\label{p1-1-8}
 There is an $\eta_0>0$, such that for any $\eta\in (0, \eta_0]$, and $z_i$ close to $b_i$,  there exists $v_{\eta, z}\in  F_{\eta,z}$ with $z=(z_1,\cdots,z_k)$,  such that
\begin{equation*}
\begin{split}
\int_{\mathbb R^N}  &\bigl( \eta^2 \nabla w_{\eta}  \nabla \psi  +  \bigl( 1+ \eta^2 V(x)\bigr)w_{\eta} \psi
 =\int_{\mathbb R^N} w_{\eta}^{3} \psi,~~
\mbox{for all}~\psi \in  F_{\eta,z},
\end{split}
\end{equation*}
where  $w_{\eta}(x)=\displaystyle\sum^k_{i=1}Q_{\eta,z_i}(x)+  v_{\eta, z}(x)$, and

\[
 F_{\eta,z}=\left \{ u(x)\in H^1(\R^N):
\Big\langle u,\frac{\partial Q_{\eta,z_{i}}(x)}{\partial{x_j}}
\Big\rangle_{\eta}=0, ~j=1,\cdots,N, \;i=1,\cdots,k
 \right\}.
\]
 Moreover, it holds
\begin{equation*}
\|v_{\eta, z}\|_\eta=O\big(
\sum^k_{i=1}\big| V(z_i)-V_i\big|\eta^{\frac{N}{2}+2}+\sum^k_{i=1}\big|\nabla V(z_i)\big|\eta^{\frac{N}{2}+3}+\eta^{\frac{N}{2}+4}\big).
\end{equation*}

\end{Prop}

 To obtain a true solution for \eqref{20-7-11}, we need to choose $z$, such that
\begin{equation*}
\int_{B_d(x_{a,i})}  \Bigl(-\eta^2 \Delta w_{\eta} \frac{\partial w_\eta}{\partial x_j}  + \bigl( 1+ \eta^2 V(x)\bigr) w_{\eta} \frac{\partial w_\eta}{\partial x_j}-w_{\eta}^{3} \frac{\partial w_\eta}{\partial x_j}
\Bigr) =0,\quad \forall \;i=1,\cdots,k,~~j=1,\cdots,N.
\end{equation*}
 It is easy to check that the above identities are  equivalent to
\begin{equation}\label{10-7-11}
\int_{B_d(x_{a,i})}  \frac{\partial V(x)}{\partial x_j} w^2_\eta=O\big(e^{-\frac{\gamma}{\eta}}\big),\quad \forall \;i=1,\cdots,k,~~j=1,\cdots,N.
\end{equation}

For $z_i$ close to $b_i$, $z_i\in \Gamma_{t,i}$ for some $t$ close to $V_i$ and $z=(z_1,\cdots,z_k)$.
In the following, we use $\nu_i$ to denote the unit normal vector of $\Gamma_{t,i}$ at $z_i$, while we use $\tau_{i,j}$ ($j=1, \cdots, N-1$),
 to denote the principal directions of $\Gamma_{t,i}$ at $x_{a,i}$. Then, at $z_i$, it holds
\[D_{\tau_{i,j}}  V(z_i)=0,~\mbox{for}~
j=1, \cdots, N-1,~\mbox{and}~ |\nabla  V(z_i)|= |D_{\nu_i}  V(z_i)|.
 \]

 We first prove the following results.

\begin{Lem}\label{lem-001-1-8}
Under the condition \textup{($V$)}, $\displaystyle\int_{B_d(x_{a,i})} D_{\nu_i}  V(x) u^2_\eta
=O\big(e^{-\frac{\gamma}{\eta}}\big) $ is equivalent to
\begin{equation}\label{00-1-1-8}
D_{\nu_i}  V(z_i)=O\bigl( \eta^2\bigr).
\end{equation}

\end{Lem}

\begin{proof}
First, we have
\begin{equation}\label{00-aluo-6}
\begin{split}
\int_{\mathbb R^N}& D_{\nu_i}  V(x)Q^2_{\eta,z_i}(x)\\=&
-2 \int_{\mathbb R^N} D_{\nu_i}  V(x)Q_{\eta,z_i}(x)v_{\eta,z}-\int_{\mathbb R^N}D_{\nu_i}  V(x)
  v^2_{\eta,z}\\=&
O\big(|  D_{\nu_i}  V(z_i)| \eta^{\frac{N}{2}} \cdot\| v_{\eta, z}\|_\eta+\eta^{\frac{N}{2}+1}\| v_{\eta, z}\|_\eta+
\| v_{\eta, z}\|^2_\eta\big)=
O\big(\eta^{N+2}\big).
\end{split}
\end{equation}
On the other hand,
we have
\begin{equation}\label{00-abcluo-3}
\begin{aligned}
\int_{\mathbb R^N}& D_{\nu_i}  V(x)Q^2_{\eta,z_i}(x)=a_*\eta^{N}  D_{\nu_i}  V( z_i)
 +O\big(\eta^{N+2}\big),
\end{aligned}
\end{equation}
Then we get \eqref{00-1-1-8} by combining \eqref{00-aluo-6} and \eqref{00-abcluo-3}.

\end{proof}

\begin{Lem}\label{lem-002-1-8}
Under the condition \textup{($V$)},  $\displaystyle\int_{B_d(x_{a,i})} D_{\tau_i}  V(x) u^2_\eta
=O\big(e^{-\frac{\gamma}{\eta}}\big) $   is equivalent to
\begin{equation}\label{01-1-1-8}
(D_{\tau_i}  \Delta V) (z_i)=O\Big(\big(\sum^k_{l=1}\big|V(z_l)-V_l\big|\big) \eta+ \eta^{2}\Big).
\end{equation}

\end{Lem}

\begin{proof}

Let  $G(x)=  \bigl\langle \nabla V(x), \tau_i\bigr\rangle$.  Then, similar to the estimate \eqref{luo-6}, we have
\begin{equation}\label{00-luo-6}
\begin{split}
\int_{\mathbb R^N}G(x)Q^2_{\eta,z_i}(x)=&
-2 \int_{\mathbb R^N} G(x)Q_{\eta,z_i}(x) v_{\eta, z}-\int_{\mathbb R^N} G(x)  v^2_{\eta, z}\\=&
O\Big( \big(\sum^k_{l=1}\big| V(z_l)-V_l\big|\big)\eta^{N+3}+ \eta^{N+4} \Big).
\end{split}
\end{equation}
On the other hand,  in view of $ G(z_i)=0$, it is easy to show
\begin{equation}\label{01-luo-6}
\int_{\mathbb R^N} G(x) Q^2_{\eta,z_i}(x)=\frac12
\eta^{N+2} \Delta G(z_i) B+O\bigl (\eta^{N+4}\bigr),
\end{equation}
where
\begin{equation}\label{luopeng10}
B=\frac{1}{N}\displaystyle\int_{\R^N}|x|^2Q^2.
\end{equation}
Thus,  \eqref{01-1-1-8} follows from  \eqref{00-luo-6} and \eqref{01-luo-6}.

\end{proof}

\begin{Thm}\label{Thm-1.a}
For $\lambda>0$ large, \eqref{1-23-5} has a solution $u_\lambda$, satisfying
\[
u_\lambda (x) =    \sum^k_{i=1} \sqrt\lambda Q\bigl( \sqrt\lambda(x-x_{\lambda,i})\bigr) +\omega_\lambda,
\]
where $x_{\lambda,i} \to b_i$, and $\int_{\mathbb R^N} \bigl(|\nabla \omega_\lambda|^2 + \omega_\lambda^2\bigr)\to 0$ as $\lambda\to +\infty$.
\end{Thm}
\begin{proof}  As pointed out earlier, we need to solve \eqref{10-7-11}.   By Lemmas~\ref{lem-001-1-8} and
\ref{lem-002-1-8}, the equation \eqref{10-7-11} is equivalent to
\begin{equation*}
D_{  \nu_i}  V(z_i)=O\bigl( \eta^2 \bigr),\quad
(D_{\tau_i}  \Delta V) (z_i)=O\Big(\big(\sum^k_{l=1}\big|V(z_l)-V_l\big|\big)\eta+ \eta^{2}\Big).
\end{equation*}
Let $\bar z_i\in \Gamma_i$ be the point such that $z_i-\bar z_i = \alpha_i \nu_i$ for some $\alpha_i\in \mathbb R$.  Then, we have $D_{\nu_i} V(\bar z_i)=0$.
As a result,
\begin{eqnarray*}
D_{  \nu_i}  V(z_i)=D_{  \nu_i}  V(z_i)-D_{  \nu_i}  V(\bar z_i)=D^2_{ \nu_i \nu_i}  V(\bar z_i) \langle z_i-\bar z_i, \nu_i\rangle + O(|z_i-\bar z_i|^2).
\end{eqnarray*}
By the non-degenerate assumption,  
we find that $D_{  \nu_i}  V(z_i)=O\bigl(\eta^2\bigr)$ is equivalent to
$
\langle z_i-\bar z_i, \nu_i\rangle=O\bigl( \eta^2+ |z_i-\bar z_i|^2\bigr)$.
This means that $
D_{\nu_i}  V(z_i)=O\bigl( \eta^2\bigr)$ can be written as
\begin{equation}\label{62-1-8}
|z_i-\bar z_i| =O\bigl( \eta^2\bigr).
\end{equation}
Let $\bar \tau_{i,j}$ be the $j$-th tangential unit vector of $\Gamma_i$ at $\bar z_i$. Now by the condition \textup{($V$)}, we have
\begin{eqnarray*}
(D_{  \tau_{i,j}}  \Delta V) (z_i)=(D_{ \bar  \tau_{i,j}}  \Delta V) (\bar z_i) +O(|z_i-\bar z_i|)=(D_{ \bar  \tau_{i,j}}  \Delta V) (\bar z_i)+O(\varepsilon^2),
\end{eqnarray*}
and
\begin{equation*}
\begin{split}
(D_{\bar \tau_{i,j}}  \Delta V) (\bar z_i)=&(D_{ \bar  \tau_{i,j}}  \Delta V) (\bar z_i)-(D_{  \tau_{i, j,0}}  \Delta V) (b_i)=\bigl\langle (\nabla_T D_{  \tau_{i, j,0}}  \Delta V) ( x_0), \bar z_i-b_i\bigr\rangle
+ O(|\bar z_i-b_i|^2),
\end{split}
\end{equation*}
where $\nabla_{T_i}$ is the tangential gradient on $\Gamma_i$ at  $b_i\in \Gamma_i$, and  $\tau_{i,j, 0}$ is the $j$-th tangential unit vector of $\Gamma_i$ at $b_i$.
Therefore,  $(D_{\tau_i}  \Delta V) (z_i)=O\Big( \big(\sum^k_{l=1}|V(z_l)-V_l|\big) \eta+ \eta^{2}\Big)$ can be rewritten as
\begin{equation}\label{63-1-8}
\bigl\langle (\nabla_T D_{  \tau_{i, j,0}}  \Delta V) ( x_0), \bar z_i-b_i\bigr\rangle= O({\eta^2}+|\bar z_i-b_i|^2).
\end{equation}
So we can solve \eqref{62-1-8} and \eqref{63-1-8} to obtain $z_i=x_{\eta,i}$ with  $x_{\eta,i} \to b_i$ as $ \eta\to 0$.
\end{proof}

\begin{proof}[\textbf{Proof of Theorem~\ref{nth1.2}}]

  Let $w_\lambda$ is a $k$-peak solution as in Theorem \ref{Thm-1.a}, and we define
\[
u_\lambda =\frac{w_\lambda}{\Bigl( \displaystyle\int_{\mathbb R^N}w_\lambda^2\Bigr)^{\frac12}}.
\]
Then $\displaystyle\int_{\mathbb R^N} u_\lambda^2=1$, and
\begin{equation}\label{12-26-10}
-\Delta u_\lambda+  V(x)u_\lambda = a_\lambda u_\lambda^3-\lambda u_\lambda, \quad  \text{in}\; \mathbb R^N,
\end{equation}
with
$
a_\lambda=\displaystyle\int_{\mathbb R^N} w_\lambda^2$.

For $N=2$, similar to \eqref{8-29-1}, under the condition  \textup{($V$)}, we can prove that
\begin{equation*}
\Bigl( ka_*- \int_{\mathbb R^2} w_\lambda^2\Bigr) \lambda^2= \frac{1}{2} \sum^k_{i=1}\Delta  V(b_i)  \displaystyle\int_{\R^2}|x|^2Q^2(x) +o\big(1\big).
\end{equation*}
 This shows that if $\displaystyle\sum^k_{i=1}\Delta  V(b_i)\ne 0$, $\displaystyle\int_{\mathbb R^N} w_\lambda^2\ne ka_*$.  Take $\lambda_0>0$ large and let
$
b_0=\displaystyle\int_{\mathbb R^2}w_{\lambda_0}^2$.
Then,
$$b_0<ka_*,~ \mbox{if}~ \displaystyle\sum^k_{i=1}\Delta  V(b_i)> 0~~ \mbox{and} ~b_0>ka_*~\mbox{if}~  \displaystyle\sum^k_{i=1}\Delta  V(b_i)< 0.$$
 By the mean value theorem, for any $a$ between $b_0$ and $ka_*$,
there exists $\lambda_a>0$ large, such that the solution $w_a$ of \eqref{1-23-5} with $\lambda=\lambda_a$ satisfies
$\displaystyle\int_{\mathbb R^2}w_\lambda^2= a$. Thus, for such $a$, we obtain a $k$-peak solution for \eqref{8-18-1}, where $\mu_a= -\lambda_a$.

For $N=3$, similar to \eqref{ab8-29-1}, we can prove
\begin{equation*}
  \sqrt{\lambda} \int_{\mathbb R^3} w_\lambda^2=ka_*+o\big(1\big), ~\mbox{as}~ a\searrow 0.
\end{equation*}
 Take $\lambda_0>0$ large and let $
b_0=\displaystyle\int_{\mathbb R^3} w_{\lambda_0}^2$.
Then, by the mean value theorem, for any $a$ between $0$ and $b_0$,
there exists $\lambda_a>0$ large, such that the solution $u_a$ of \eqref{1-23-5} with $\lambda=\lambda_a$ satisfies
$\displaystyle\int_{\mathbb R^3} w_\lambda^2= a$. Thus, for such $a$, we obtain a $k$-peak  solution for \eqref{8-18-1}, where $\mu_a= -\lambda_a$.

\end{proof}

 Before we end this section, we discuss briefly the existence of clustering $k$-peak solutions for
\eqref{8-18-1}--\eqref{8-28-3}.  The function $\Delta V(x)|_{x\in \Gamma_{i}}$ has a minimum point and a maximum point.  Let us assume that  $\Delta V(x)|_{x\in \Gamma_{i}}$ has an
isolated maximum point $b_i\in \Gamma_{i}$. That is, we assume that $\Delta V(x)< \Delta V(b_i)$ for all $x\in \Gamma_{i} \bigcap (B_\delta (b_i)\setminus \{b_{i}\})$.  We now use
$ \sum_{j=1}^kQ_{\eta,x_{\eta,j}}$
as an approximate solution of \eqref{20-7-11},  where $x_{\eta, j}$ satisfies
 $x_{\eta, j}\to b_i$, $j=1,\cdots, k$,  $\frac{|x_{\eta, j}- x_{\eta, l}|}\eta \to +\infty$, $l\ne j$,  as $\eta \to 0.$
We have the following existence result for \eqref{20-7-11}.
\begin{Prop}\label{Thm1.b}
 Assume that \textup{($V$)} holds,  and $\frac{\partial^2 V(\bar x)}{\partial \nu_{i}^2}\ne 0$ for any $\bar x\in \Gamma_{i}$ with some $i\in \{1,\cdots,k\}$.   If   $b_i\in \Gamma_{i}$ is an isolated maximum point of   $ \Delta V(x)|_{x\in \Gamma_{i}}$ on $\Gamma_{i}$, then
  for any integer $k>0$, there exists an $\eta_0>0$, such that for any $\eta\in (0, \eta_0]$,  problem \eqref{20-7-11} has a solution $u_\eta$, satisfying
\[
u_\eta (x) =\sum_{j=1}^kQ_{\eta,x_{\eta,j}} +\omega_\eta,
\]
where  $x_{\eta, j}\to b_i$, $j=1,\cdots, k$,  $\frac{|x_{\eta, j}- x_{\eta, l}|}\eta \to +\infty$, $l\ne j$,  and $\displaystyle\int_{\mathbb R^N} \bigl(
\eta^2 |\nabla \omega_\eta|^2 + \omega_\eta^2\bigr)= o(\eta^{\frac N2})$ as $\eta\to 0$.
\end{Prop}

\begin{proof}

Define
\[
I(u) = \frac12 \int_{\mathbb R^N} \bigl( \eta^2 |\nabla u|^2 +
(1+ \eta^2V(x) )u^2\bigr) -\frac14 \int_{\mathbb R^N} u^4.
\]
We have the following energy expansion:

\begin{equation}\label{100-33-3-1}
\begin{split}
&I\bigl(  \sum_{j=1}^kQ_{\eta,x_{\eta,j}} \bigr)
\\=& k  E \eta^N + E \eta^N\sum_{j=1}^k  \frac{\partial^2 V(\bar x_{\eta,j})}{\partial \nu_{i,j}^2} r_{\eta,j}^2+F\eta^{N+4} \sum_{j=1}^k \Delta V(\bar x_{\eta,j})\\
& - \sum_{j>m}(a_0+o(1))\eta^{N}   e^{- \frac{|x_{\eta,m} -x_{\eta,j}|}{\eta}} \Bigl( \frac{\eta}{ |x_{\eta,m} -x_{\eta,j}|}
 \Bigr)^{\frac{N-1}{2}}
+O(\eta^{N+5}+\eta^{N+2}r_{\eta,j}^3),
 \end{split}
\end{equation}
where $a_0>0$, $E= \frac14
\displaystyle\int_{\mathbb R^N}Q^{4}>0$, $F>0$ and $r_{\eta,j}= |x_{\eta,j}-\bar x_{\eta,j}|$, $\bar x_{\eta,j}\in \Gamma_{i}$ is the point such that $|x_{\eta,j}-\bar x_{\eta,j}|= d(x_{\eta,j}, \Gamma_{i})$.
In fact,
\begin{equation}\label{33-3-1}
\begin{split}
I\bigl(  \sum_{j=1}^kQ_{\eta,x_{\eta,j}} \bigr)
=& E\sum_{j=1}^k  V(x_{\eta, j}) \eta^N+  F \eta^{N+4} \sum_{j=1}^k \Delta V(x_{\eta, j})\\& - \sum_{j>m}(a_0+o(1))\eta^{N}   e^{- \frac{|x_{\eta,m} -x_{\eta,j}|}{\eta}} \Bigl( \frac{\eta}{ |x_{\eta,m} -x_{\eta,j}|}
 \Bigr)^{\frac{N-1}{2}}
+O(\eta^{N+5}).
 \end{split}
\end{equation}
 Also, we have
\begin{equation}\label{c33-3-1}
 V(x_{\eta,j}) = 1 +\frac{\partial^2 V(\bar x_{\eta,j})}{\partial \nu_{i,j}^2} r_{\eta,j}^2 + O(r_{\eta,j}^3) ~
\mbox{and}~
\Delta V(x_{\eta,j})=\Delta V(\bar x_{\eta,j})+ O(r_{\eta,j}).
\end{equation}
So,  \eqref{100-33-3-1}  follows from  \eqref{33-3-1} and \eqref{c33-3-1}.

To obtain a solution $u_{\eta} $ of the form    $\sum_{j=1}^kQ_{\eta,x_{\eta,j}} +\omega_{\eta }$,   we can first carry out the reduction argument as in Proposition~\ref{p1-1-8} to obtain $\omega_{\eta }$, satisfying

\begin{equation}\label{nnn100-33-3-1}
\|\omega_{\eta}\|_\eta=\eta^{N+2} O\big(\sum_{j=1}^k|\nabla V(x_{\eta,j})|\eta+\eta^{2}+\sum_{j\ne m}  e^{- \frac{ (1+\sigma)|x_{\eta,m} -x_{\eta,j}|}{2\eta}  }\big),
\end{equation}
for some $\sigma>0$.  Define

\[
K(x_{\eta,1},\cdots, x_{\eta,k})=I\bigl(Q_{\eta,x_{\eta,j}} +\omega_{\eta}\bigr).
\]
Then, it follows from \eqref{nnn100-33-3-1} that we can obtain the same expansion \eqref{100-33-3-1}  for $K(x_{\eta,1},\cdots, x_{\eta,k})$.
Now we set
\[
M= \bigl\{\bigl( r, \bar x):  r\in (-\delta \eta^2, \delta \eta^2),\; \bar x\in  B_\delta(b_i)\cap \Gamma_i\bigr\}.
\]

If $\Gamma_i$ is a local maximum set of $V(x)$  and $\frac{\partial^2 V(\bar x)}{\partial \nu_i^2}<0$ for any $\bar x\in \Gamma_i$, then it is easy to
prove that $K(x_{\eta,1},\cdots, x_{\eta,k})$
has a critical point, which is a maximum point of $K$ in $$S_{\eta, k}:=\bigl\{ (x_{\eta,1},\cdots, x_{\eta,k}):  x_{\eta,j}\in M,
 |x_{\eta, j}-x_{\eta, m}|> \theta \eta
|\ln \eta|,\, m\ne j\bigr\},$$  where   $\theta>0$ is some constant.

If $\Gamma_i$ is a local minimum set of $V(x)$  and $\frac{\partial^2 V(\bar x)}{\partial \nu_i^2}>0$ for any $\bar x\in \Gamma_i$, then  $$E \frac{\partial^2 V(\bar x)}{\partial \nu_i^2}r^2+  F \eta^{2}  \Delta V(\bar x)$$ has
a saddle point $(0, b_i)$ in $M$.  We can use a topological argument as in \cite{DY10,DY11} to prove that   $K(x_{\eta,1},\cdots, x_{\eta,k})$
has a critical point  in $S_{\eta, k}$.
\end{proof}

Similar to the proofs of Theorem \ref{nth1.2}, from Proposition \ref{Thm1.b}, we have the following result:

\begin{Thm}\label{nnth1.2}
 Assume that \textup{($V$)} holds,  and $\frac{\partial^2 V(\bar x)}{\partial \nu_{i}^2}\ne 0$ for any $\bar x\in \Gamma_{i}$ and some ${i}\in \{1,\cdots,k\}$.   If   $b_{i}\in \Gamma_i$ is an isolated maximum point of   $ \Delta V(x)|_{x\in \Gamma_{i}}$ on $\Gamma_{i}$, then
  for any integer $k>0$, \eqref{8-18-1}--\eqref{8-28-3} has a  solution satisfying
\[
u_a = \sqrt{\frac{-\mu_a}{ a}}\Bigl( \sum_{j=1}^k
Q\big(\sqrt{-\mu_a}(x-x_{a,j})\big)+\omega_a (x)\Bigr),
\]
with
$$\mu_a\rightarrow-\infty,
~x_{a, j}\to b_{i},~j=1,\cdots, k,~\frac{|x_{a, j}- x_{a, l}|}{\sqrt{-\mu_a}} \to +\infty,~ l\ne j,$$
and
$\displaystyle\int_{\R^N}\big[-\frac{1}{\mu_a}|\nabla \omega_a|^2+  \omega^2_a\big]=o\big((\frac{1}{\sqrt{-\mu_a}})^{N}\big)$
as $a\rightarrow k a_*$ if $N=2$,  or $a\searrow 0$ if $N=3$.
\end{Thm}
\section{Local uniqueness}

From Lemma \ref{lem-5-05-1}, a $k$-peak solution $\tilde u_a$ can be written as
\begin{equation}\label{a-5-21-3}
\tilde{u}_a(x)= \sum^k_{i=1}Q_{\varepsilon,x_{a,i}}+v_{a}(x),
\end{equation}
with
$|x_{a,i}-b_i|=o(1)$,
$\varepsilon =\frac{1}{ \sqrt{-\mu_a}}$, $v_a \in \displaystyle\bigcap^k_{i=1}E_{a,{x}_{a,i}}$ and
\begin{equation}\label{06-09-2}
\|v_{a}\|_a=O\Big( \displaystyle\sum^k_{i=1}\big|V(x_{a,i})-V_i\big|\varepsilon^{\frac{N}{2}+2} + \displaystyle\sum^k_{i=1}\big|\nabla V(x_{a,i})\big|\varepsilon^{\frac{N}{2}+3}
+\varepsilon^{\frac{N}{2}+4}\Big).
\end{equation}
Also we know  $x_{a,i}\in \Gamma_{t_a,i}$ for some $t_a\to V_i$.
Similar to the last section, we use $\nu_{a,i}$ to denote the unit normal vector of $\Gamma_{t_a,i}$ at $x_{a,i}$, while we use $\tau_{a,i,j}$
 to denote the principal direction  of $\Gamma_{t_a,i}$ at $x_{a,i}$. Then, at $x_{a,i}$, it holds
\begin{equation}\label{06-09-3}
 D_{\tau_{a,i,j}}  V(x_{a,i})=0,\quad \big|\nabla  V(x_{a,i})\big|= \big|D_{  \nu_{a,i}}  V(x_{a,i})\big|.
\end{equation}

 We first prove the following result.

\begin{Lem}\label{lem-1-1-8}
Under the condition \textup{($V$)}, we have
\begin{equation}\label{1-1-8}
D_{  \nu_{a,i}}  V(x_{a,i})=O\bigl(\varepsilon^2\bigr),~\mbox{for}~i=1,\cdots,k.
\end{equation}

\end{Lem}

\begin{proof}

 We use \eqref{30-31-7} to obtain
\begin{equation}\label{3-1-8}
\int_{B_{d}(x_{a,i})} D_{\nu_{a,i}}  V(x) \tilde{u}^2_a
=O(e^{-\frac{\gamma}{\varepsilon}}).
\end{equation}
Then by \eqref{a-5-21-3}--\eqref{06-09-3} and \eqref{3-1-8}, we get
\begin{equation}\label{aluo-6}
\begin{split}
\int_{B_{d}(x_{a,i})} & D_{\nu_{a,i}}  V(x)
Q^2_{\varepsilon,x_{a,i}}\\=&
-2 \int_{B_{d}(x_{a,i})} D_{  \nu_{a,i}}  V(x)
Q_{\varepsilon,x_{a,i}} v_{a}-\int_{B_{d}(x_{a})}D_{\nu_{a}}  V(x)
 v^2_{a}+O\big(e^{-\frac{\gamma}{\varepsilon}}\big)\\=&
O\big(|  D_{  \nu_{a,i}}  V( x_{a,i})|\varepsilon^{\frac{N}{2}} \cdot\|v_{a}\|_a+\varepsilon^{\frac{N}{2}+1}\|v_{a}\|_a+
\|v_{a}\|^2_a\big)+O\big(e^{-\frac{\gamma}{\varepsilon}}\big)\\
=&
O\big(\sum^k_{i=1}\big|V( x_{a,i})-V_i\big|\varepsilon^{\frac{N}{2}+2}+ |D_{  \nu_{a,i}}  V( x_{a,i})|\varepsilon^{\frac{N}{2}+3}+\varepsilon^{\frac{N}{2}+4}\big).
\end{split}
\end{equation}
On the other hand, by Taylor's expansion,
we have
\begin{equation}\label{abcluo-3}
\begin{aligned}
\int_{B_{d}(x_{a,i})}& D_{  \nu_{a,i}}  V(x)
Q^2_{\varepsilon,x_{a,i}}=\varepsilon^{N} \big[(1+V_i\varepsilon^2)a_* D_{  \nu_{a,i}}  V( x_{a,i})
  +\frac{B\varepsilon^{2}}{2} \Delta  D_{  \nu_{a,i}}  V(x_{a,i}) +O\big(\varepsilon^{4}\big)\big],
\end{aligned}
\end{equation}
where $B$ is the constant in \eqref{luopeng10}.
And then  \eqref{1-1-8}  follows from \eqref{aluo-6} and \eqref{abcluo-3}.

\end{proof}

Let $\bar  x_{a,i}\in \Gamma_i$ be the point such that $ x_{a,i}-\bar  x_{a,i} = \beta_{a,i} \nu_{a,i}$ for some $\beta_{a,i}\in \mathbb R$ and $i=1,\cdots,k$.
Then we can prove

\begin{Lem}
Under the condition  \textup{($V$)}, we have
\begin{equation}\label{8-22-3}
\begin{cases}
 \bar x_{a,i}-b_i= L_i \varepsilon^2 +O(\varepsilon^{4}),\vspace{2mm}\\
 x_{a,i}-\bar  x_{a,i} = -\displaystyle\frac{B}{2a_*}\frac{\partial  \Delta V(b_i)}{\partial \nu_i} \Big(\frac{\partial^2 V(b_i)}{\partial \nu_i^2 } \Big)^{-1} \varepsilon^{2}\nu_{a,i} +
 O(\varepsilon^{4}),
\end{cases}
\end{equation}
where $B$  is the constant in \eqref{luopeng10} and  $L_i$ is a  vector  depending  on $b_i$ and $i=1,\cdots,k$.
\end{Lem}
\begin{proof}

  It follows from  \eqref{aluo-6} and \eqref{abcluo-3} that
\begin{equation}\label{bluo-3}
\begin{split}
 \big(a_*+&O(\varepsilon^{2})\big)  D_{  \nu_{a,i}}  V( x_{a,i})
  +\frac{B\varepsilon^2}{2}\Delta  D_{  \nu_{a,i}}  V(x_{a,i})
  \\&=O\big(\varepsilon^4+\varepsilon^2
  \sum^k_{i=1}|V(x_{a,i})-V_i|\big)=O\big(\varepsilon^4+\varepsilon^2 \sum^k_{i=1}|x_{a,i}-\bar x_{a,i}|^2\big).
  \end{split}
\end{equation}
Since $\frac{\partial^2V(b_i)}{\partial \nu_i ^2}\neq 0$, the outward unit normal vector $\nu_{a,i}(x)$ and   the  tangential unit vector $\tau_{a,i}(x)$  of  $\Gamma_{t_a,i}$  at
$x_{a,i}$ are Lip-continuous   in $W_{\delta,i}$, then from \eqref{bluo-3}, we find
\begin{equation}\label{2-13-8}
x_{a,i}-\bar  x_{a,i}
=-\frac{B}{2a_*}\big( \Delta  D_{  \nu_i}  V (b_i) \big) \Big(\frac{\partial^2V(b_i)}{\partial \nu_i^2} \Big)^{-1} \varepsilon^{2}+O\big(\varepsilon^{4}+\varepsilon^{2}
\big|\bar  x_{a,i}-b_i\big|\big).
\end{equation}
Then \eqref{06-09-2} and \eqref{2-13-8} implies \begin{equation}\label{06-09-4}
\|v_{a}\|_a=O\Big( \displaystyle\sum^k_{i=1}|x_{a,i}-\bar  x_{a,i}|\varepsilon^{\frac{N}{2}+2}
+\varepsilon^{\frac{N}{2}+4}\Big)=O\big(\varepsilon^{\frac{N}{2}+4}\big).
\end{equation}
Recall   $G(x)=  \bigl\langle \nabla V(x), \tau_{a,i,j}\bigr\rangle$. Then
$
G(x_{a,i}) =0.
$
Similar to \eqref{luo-6} and \eqref{aluo-6},  we have
\begin{equation}\label{06-09-5}
\begin{split}
\int_{B_{d}(x_{a,i})}&G(x)Q^2_{\varepsilon,x_{a,i}}\\=&
-2 \int_{B_{d}(x_{a,i})} G(x) Q_{\varepsilon,x_{a,i}}v_{a}-\int_{B_{d}(x_{a,i})} G(x) v^2_{a}+O\big(e^{-\frac{\gamma}{\varepsilon}}\big)\\=&
-2 \int_{B_{d}(x_{a,i})} G(x) Q_{\varepsilon,x_{a,i}} v_{a}+O\bigl(\|v_{a}\|^2_a\big)+O\big(e^{-\frac{\gamma}{\varepsilon}}\big)\\
=&-2 \int_{B_{d}(x_{a,i})} \langle \nabla G(x_{a,i}), x-x_{a,i}\rangle Q_{\varepsilon,x_{a,i}} v_{a}+O\big(\varepsilon^{N+6}\big).
\end{split}
\end{equation}
On the other hand,  in view of
$\nabla V(x)=0$,  $ x\in \Gamma_i,$
we find
\begin{eqnarray}\label{06-09-6}
 \nabla G(x_{a,i})=\bigl\langle \nabla^2 V(x_{a,i}), \tau_{a,i,j}\bigr\rangle= \bigl\langle \nabla^2 V(\bar x_{a,i}), \bar \tau_{a,i,j}\bigr\rangle+
 O\big(|x_{a,i}-\bar x_{a,i}|\big)=  O\big(|x_{a,i}-\bar x_{a,i}|\big),
\end{eqnarray}
where  $\bar  x_{a,i}\in \Gamma_i$ is the point such that $ x_{a,i}-\bar x_{a,i} = \beta_{a,i} \nu_{a,i}$ for some $\beta_{a,i}\in \mathbb R$, and $\bar \tau_{a,i,j}$ is the tangential vector of $\Gamma_i$ at $\bar x_{a,i}\in \Gamma_i$.
Therefore, from \eqref{2-13-8}, \eqref{06-09-4} and \eqref{06-09-6}, we know
 \begin{equation}\label{06-09-7}
  \begin{split}
  \int_{B_{d}(x_{a,i})}& \langle \nabla G(x_{a,i}), x-x_{a,i}\rangle  Q\Big(\frac{x-x_{a,i}}\varepsilon \Big) v_{a}\\
  =&O\bigl( \varepsilon^{\frac{N}{2}+1}  |\nabla G(x_{a,i})| \|v_{a}\|_a\bigr)
  = O\big(|x_{a,i}-\bar x_{a,i}| \varepsilon^{N+5}\big)=O\big(\varepsilon^{N+7}\big).
  \end{split}
  \end{equation}
Then by \eqref{06-09-5} and  \eqref{06-09-7}, we find
\begin{equation}\label{06-09-8}
\begin{split}
\int_{B_{d}(x_{a,i})}G(x)Q^2_{\varepsilon,x_{a,i}}=O\big(\varepsilon^{N+6}\big).
\end{split}
\end{equation}
On the other hand, by the Taylor's expansion,
we can prove
\begin{equation}\label{06-09-9}
\begin{aligned}
\int_{B_{d}(x_{a,i})}G(x)Q^2_{\varepsilon,x_{a,i}}=
\big[\frac{B\varepsilon^{N+2}}{2} (1+V_i\varepsilon^2)\big](D_{\tau_{a,i}} \Delta V)(x_{a,i})
+\frac{ H_{\tau_i} \varepsilon^{N+4}}{24}+O\big(\varepsilon^{N+6}\big),
\end{aligned}
\end{equation}
where
$$H_{\tau_i} =\sum^2_{l=1}\sum^2_{m=1} \frac{\partial^4 G(b_i)}{ \partial x^2_l \partial x^2_m}\int_{\R^N}x_l^2x^2_mQ^2.$$
So \eqref{06-09-8} and \eqref{06-09-9} give
\begin{equation}\label{06-09-10}
\begin{aligned}
  (D_{\tau_{a,i}} \Delta V)(x_{a,i})
=-\frac{ H_{\tau_i} \varepsilon^{2}}{1 2B}+O\big(\varepsilon^{4}\big).
\end{aligned}
\end{equation}
We denote by $\bar \tau_{a,i}$ the tangential vector of $\Gamma_i$ at $\bar x_{a,i}$. Then  by \eqref{2-13-8}, we get
\begin{equation}\label{06-09-11}
\begin{split}
(D_{\tau_{a,i}} \Delta V)(x_{a,i})=&(D_{\bar \tau_{a,i}} \Delta V)(\bar x_{a,i})+ \langle A_{\tau_i}, x_{a,i}- \bar x_{a,i}\rangle+O(|
x_{a,i}- \bar x_{a,i}|^2)\\=&
(D_{\bar \tau_{a,i}} \Delta V)(\bar x_{a,i})+ B_{\tau_i} \varepsilon^2 + O(\varepsilon^4),
\end{split}
\end{equation}
where $A_{\tau_i}$ is a vector depending on $b_i$  and $B_{\tau_i}$ is a constant  depending on $b_i$.  Moreover,
\begin{equation}\label{06-09-12}
(D_{\bar \tau_{a,i}} \Delta V)(\bar x_{a,i})= \Big(D^2_{\tau_i}(\Delta V)(b_i)\Big) (\bar x_{a,i}-b_i)+ O(|\bar x_{a,i}-b_i|^2).
\end{equation}
Therefore, from  \eqref{06-09-10}--\eqref{06-09-12}, we find
\begin{equation}\label{abc-lll}
\begin{split}
& D^2_{\tau_i}(\Delta V)(b_i) (\bar x_{a,i}-b_i)
=-(\frac{ H_{\tau_i} }{1 2B}+B_{\tau_i})\varepsilon^{2}+O\big(\varepsilon^4\big)+ O(|\bar x_{a,i}-b_i|^2).
\end{split}
\end{equation}
Since  $D^2_{\tau_i} (\Delta V)(b_i) $ is non-singular, we can complete the proofs of \eqref{8-22-3} from \eqref{2-13-8} and \eqref{abc-lll}.
\end{proof}

Let $$
\delta_a:=
\begin{cases}
|ka_*-a|^{\frac{1}{4}}|\beta_1|^{-\frac14}B^{\frac14},~& \mbox{for}~N=2,\\
a(ka_*)^{-1},~& \mbox{for}~N=3,
\end{cases}
$$
where  $\beta_1= \displaystyle\sum^k_{i=1} \Delta V(b_i) $ and   $B$  is the constant in \eqref{luopeng10}.
\begin{Prop}
Under the condition \textup{($V$)}, for $N=2,3$, it holds
\begin{equation}\label{abc8-29-1}
-\mu_a\delta_a^2= 1+\gamma_1\delta_a^2+O\big(\delta_a^4\big),
\end{equation}
and
\begin{equation} \label{8-27-38}
 x_{a,i}-b_i= \bar L_i \delta_a^2 +O(\delta_a^{4}),~\mbox{for}~i=1,\cdots,k.
\end{equation}
where $\gamma_1$ and the vector $\bar L_i$ are constants.
\end{Prop}
\begin{proof}
First, \eqref{ab8-29-1} shows that \eqref{abc8-29-1} holds for the case $N=3$.

For $N=2$, from \eqref{8-22-3}, we get $x_{a,i}-b_i= -\tilde{L}_i\frac{1}{\mu_a}+O(\frac{1}{\mu_a^2})$ for some vector $\tilde{L}_i$.
Also from \eqref{lt1}, we know $
\|v_{a}\|_a=O\big(\varepsilon^{5}\big)$. we can calculate  \eqref{bcaaa8-19-3}--\eqref{8-27-21} more precise, which will gives us
\begin{equation}\label{ac8-27-21}
\begin{split}
\int_{\R^2}u^4_a
=-\frac{2ka_*\mu_a}{a^2}+\frac{b}{\mu^2_a}+O\Big(-\frac{1}{\mu_a^3}\Big),
\end{split}
\end{equation}
and
\begin{equation}\label{abc8-19-3}
\begin{split}
\mbox{LHS of}~\eqref{8-19-3}
=
-\frac{1}{a\mu_a} \Delta V(b_j) \int_{\R^2}|x|^2Q^2(x)+ \frac{\tilde b}{\mu^2_a}
+O\big(-\frac{1}{\mu^3_a}\big),
\end{split}
\end{equation}
where $b$ and $\tilde{b}$ are some constants.
Then from \eqref{8-19-3}, \eqref{8-27-21},  \eqref{ac8-27-21} and \eqref{abc8-19-3}, we get \eqref{abc8-29-1}.

Finally, we can find \eqref{8-27-38}  by \eqref{8-22-3} and \eqref{abc8-29-1}.
\end{proof}
 By a change of variable, the problem
\eqref{8-18-1}--\eqref{8-28-3} can be changed into the following problem
\begin{equation}\label{1-12-11}
-\delta_a^2\Delta u+ \bigl( -\mu_a \delta_a^2 + \delta_a^2 V(x)\bigr)u= u^3,~
u\in H^1(\R^N),
\end{equation}
and
\begin{equation}\label{a1-12-11}
\int_{\R^N} u^2=a\delta_a^2.
\end{equation}
Then
similar to Lemma \ref{lem-5-05-1}, the $k$-peak solution of \eqref{1-12-11}--\eqref{a1-12-11} concentrating at $b_1,\cdots,b_k$ can be written as $
  \displaystyle \sum^k_{i=1}\tilde{Q}_{\delta_a,x_{a,i}}+\bar v_{a}(x)$,
with $
|x_{a,i}-b_i|=o(1)$,  $\|\bar v_{a}\|_{\delta_a}=o(\delta_a^{\frac{N}{2}})$,
 and
\begin{equation*}
\begin{split}
{\bar v_a} \in \bigcap^k_{i=1} \tilde{E}_{a,{x}_{a,i}}&:=\left \{ v\in H^1(\R^N):
\Big\langle v,\frac{\partial \tilde{Q}_{\delta_a,x_{a,i}}}{\partial{x_j}}\Big\rangle_{a}=0, ~j=1,\cdots,N
 \right\},
 \end{split}
 \end{equation*}
where  $\tilde{Q}_{\delta_a,x_{a,i}}:=
Q\Big(\frac{\sqrt{1+(\gamma_1+V_i)\delta_a^2}
(x-x_{a,i})}{\delta_a}\Big)$,  $\|v\|^2_{\delta_a}:=\displaystyle\int_{\mathbb R^N} \bigl(  \delta_a^{2} |\nabla v|^2 +v^2\bigr)$
and $\gamma_1$ is the constant in \eqref{abc8-29-1}.
Then we can write the equation \eqref{1-12-11} as follows:
 \begin{equation*}
    L_a(\bar  v_a)
    = N_a\big(\bar  v_{a}\big)+\bar l_{a}(x),\end{equation*}
    where $N_a,L_a$ are defined by \eqref{06-07-1} and \eqref{06-07-2}. And
 \begin{equation}\label{7-15-11}
 \bar l_{a}= \big(-\mu_a \delta_a^2-1+\delta_a^2V(x)\big) \sum_{i=1}^k \tilde{Q}_{\delta_a,x_{a,i}}
+ \Big(\sum_{i=1}^k \tilde{Q}_{\delta_a,x_{a,i}}\Big)^3-
    \sum_{i=1}^k \tilde{Q}^3_{\delta_a,x_{a,i}}.
\end{equation}

\begin{Lem}It holds
 \begin{equation}\label{8-27-26}
\|\bar v_a\|_{\delta_a}=O\big(\delta_a^{\frac{N}{2}+4}\big).
\end{equation}
\end{Lem}
\begin{proof} The proofs are similar to that of Lemma \ref{lem-5-05-1},
the difference is
 \begin{equation}\label{06-07-4}
   \|\bar l_{a}\|_{\delta_a}=O\Bigl( \big(\sum^k_{i=1}\big| V(x_{a,i})-V_i\big|\big)\delta_a^{\frac{N}{2}+2}+\sum^k_{i=1}\big|\nabla V(x_{a,i})\big|\delta_a^{\frac{N}{2}+3}+ \delta_a^{\frac{N}{2}+4}\Bigr)
   =O\Bigl(\delta_a^{\frac{N}{2}+4}\Bigr).
 \end{equation}
Here the definition of $\bar l_{a}$ lies in \eqref{7-15-11}. Finally,  \eqref{06-07-4} and  \eqref{B.2} imply \eqref{8-27-26}.
\end{proof}

Let $u_a^{(1)}$ and $u_a^{(2)}$ be two $k$-peak solutions of \eqref{1-12-11}--\eqref{a1-12-11} concentrating  at $k$ points $b_1,\cdots,b_k$, which can be written as
\begin{equation}\label{8-20-1}
u_a^{(l)}=
\sum^k_{i=1}\tilde Q_{\delta_a,x_{a,i}^{(l)}}+v^{(l)}_{a}(x),
~\mbox{for}~l=1,2,~\mbox{and}~v^{(l)}_{a}\in    \bigcap^k_{i=1} \tilde E_{a,{x}^{(l)}_{a,i}}.
\end{equation}
Now we set $
\xi_{a}(x)=\frac{u_{a}^{(1)}(x)-u_{a}^{(2)}(x)}
{\|u_{a}^{(1)}-u_{a}^{(2)}\|_{L^{\infty}(\R^N)}}$.
Then $\xi_{a}(x)$ satisfies $\|\xi_{a}\|_{L^{\infty}(\R^N)}=1$. And from \eqref{1-12-11}, we find that
$\xi_a$ satisfies
\begin{equation*}
-\delta_a^2 \Delta \xi_{a}(x)+ C_{a}(x)\xi_{a}=g_a(x),
\end{equation*}
where
\begin{equation*}
C_{a}(x)=\delta_a^2 V(x)-\delta_a^2 \mu_a^{(1)}-  \big(\sum^2_{l=1}(u_a^{(l)})^2 +u_a^{(1)}u_a^{(2)}\big),~~
g_{a}(x)=\frac{\delta_a^2 (\mu_{a}^{(1)}-\mu_{a}^{(2)})}
{\|u_{a}^{(1)}-u_{a}^{(2)}\|_{L^{\infty}(\R^N)}}u_a^{(2)}.
\end{equation*}
Also, similar to \eqref{2--5}, for any fixed $R\gg 1$, there exist some $\theta>0$ and $C>0$, such that
\begin{equation}\label{a2---5}
|\xi_a(x)|+|\nabla \xi_a(x)|\leq C\sum^k_{i=1}e^{-\theta |x-x_{a,i}|/\delta_a},~\mbox{for}~
x\in \R^N\backslash \bigcup^k_{i=1}B_{R \delta_a}(x_{a,i}).
\end{equation}
Now let $\bar \xi_{a,i}(x)=\xi_{a}\big(\frac{\delta_a x+x^{(1)}_{a,i}}{\sqrt{1+(\gamma_1+V_i)\delta_a^2}}\big)$,
for $i=1,\cdots,k$,  we have
\begin{equation}\label{8-28-11}
-\Delta \bar \xi_{a,i}(x)+ \frac{C_{a}(\delta_a x+x^{(1)}_{a,i})}{1+(\gamma_1+V_i)\delta_a^2}\bar\xi_{a,i}=
\frac{g_a(\delta_a x+x^{(1)}_{a,i})}{1+(\gamma_1+V_i)\delta_a^2}.
\end{equation}
\begin{Lem}\label{lem---1}
 For $x\in B_{d\delta_a^{-1}}(0)$, it holds
\begin{equation}\label{8-28-12}
\frac{C_{a}(\delta_a x+x^{(1)}_{a,i})}{1+(\gamma_1+V_i)\delta_a^2}=
1 -3Q^2(x)+O\Big(\delta_a^4+\sum^2_{l=1}
 v_a^{(l)}(\delta_a x+x^{(1)}_{a,i})
 \Big),
\end{equation}
and
\begin{equation}\label{8-28-13}
\begin{split}
\frac{g_a(\delta_a x+x^{(1)}_{a,i})}{1+(\gamma_1+V_i)\delta_a^2}=-\frac{2}{ka_*}Q(x)\sum^k_{l=1}\int_{\R^N}Q^3(x)\bar \xi_{a,l}(x)+O\Big(\delta_a^4+\sum^2_{l=1}
 v_a^{(l)}(\delta_a x+x^{(1)}_{a,i})
 \Big).
\end{split}\end{equation}
\end{Lem}

\begin{proof}
First, \eqref{8-28-12}  can be deduced by \eqref{abc8-29-1} and \eqref{8-27-38} directly.
 Now we prove \eqref{8-28-13}.

From \eqref{1-12-11} and \eqref{a1-12-11}, for $l=1,2$, we find
\begin{equation*}
\begin{split}
 a\mu^{(l)}_a \delta_a^4= \delta_a^2\int_{\R^N} \big( |\nabla u_a^{(l)}|^2+ V(x)(u_a^{(l)})^2\big)- \int_{\R^N}(u_a^{(l)})^4,
\end{split}\end{equation*}
which gives
\begin{equation}\label{06-10-2}
\begin{split}
&\frac{a\delta_a^4(\mu_{a}^{(1)}-\mu_{a}^{(2)})}
{\|u_{a}^{(1)}-u_{a}^{(2)}\|_{L^{\infty}(\R^N)}} \\=&
 \mu_a^{(2)}\delta_a^{2}\int_{\R^N}
(u_a^{(1)}+u_a^{(2)}) \xi_a+\delta_a^2\int_{\R^N}  \big(  \nabla (u_a^{(1)}+u_a^{(2)})\cdot
 \nabla \xi_a + V(x)(u_a^{(1)}+u_a^{(2)}) \xi_a\big)\\&
-\int_{\R^N}
(u_a^{(1)}+u_a^{(2)}) \big((u_a^{(1)})^2+(u_a^{(2)})^2\big) \xi_a\\=&
\bigl( \mu^{(2)}_a -\mu_a^{(1)}\bigr)\delta_a^2\int_{\R^N}  u_a^{(1)}\xi_a
- \int_{\R^N}
(u_a^{(1)}+u_a^{(2)})  u_a^{(1)}u_a^{(2)}\xi_a,
\end{split}\end{equation}
here we use the following identity:
\[
\int_{\R^N}
 \bigl( u_a^{(1)}+ u_a^{(2)}\bigr)\xi_a= \frac{1}
{\|u_{a}^{(1)}-u_{a}^{(2)}\|_{L^{\infty}(\R^N)}}\Big(\int_{\R^N}
 ( u_a^{(1)} )^2 -\int_{\R^N}
 ( u_a^{(2)} )^2\Big) =0.
\]
Then from \eqref{abc8-29-1}, \eqref{8-27-38}, \eqref{8-27-26} and \eqref{06-10-2}, we know
\begin{equation}\label{06-10-1}
\begin{split}
&\frac{\delta_a^2 (\mu_{a}^{(1)}-\mu_{a}^{(2)})}
{\|u_{a}^{(1)}-u_{a}^{(2)}\|_{L^{\infty}(\R^N)}}=-
\frac{1}{ka_*\delta_a^N} \int_{\R^N}
(u_a^{(1)}+u_a^{(2)})  u_a^{(1)}u_a^{(2)}\xi_a+O \Big(\delta_a^4\Big).
\end{split}\end{equation}
So we can find \eqref{8-28-13}  by
\eqref{06-10-1}.
\end{proof}

Then from Lemma \ref{lem---1}, we have the following result:
\begin{Lem}
  From $|\bar \xi_{a,i}|\le 1$, we suppose that  $\bar \xi_{a,i}(x)\rightarrow \xi_i(x)$ in $C^1_{loc}(\R^N)$. Then  $\xi_i(x)$ satisfies following system:
  \begin{equation*}
-\Delta \xi_{i}(x)+\big(1-3Q^2(x)\big)\xi_{i}(x)=-\frac{2}{ka_*}Q(x)
\Big(\sum^k_{l=1}\int_{\R^N}Q^3(x)\xi_l(x)\Big),~\mbox{for}~i=1,\cdots,k.
\end{equation*}

\end{Lem}
To prove $\xi_i=0$,
 we write
\begin{equation}\label{aaaaaluo--1}
\bar \xi _{a,i}(x)= \sum^N_{j=0}\gamma_{a,i,j}\psi_j+\tilde{\xi}_{a,i}(x),~\mbox{in} ~H^1(\R^N),
\end{equation}
where $\psi_j (j=0,1,\cdots,N)$ are the functions in \eqref{aaaaa} and $\tilde{\xi}_{a,i}(x)\in \tilde{E}$ with
\begin{equation*}
\tilde{E}=\{u\in H^1(\R^N), \langle u,\psi_j\rangle=0, ~\mbox{for}~j=0,1,\cdots,N\}.
\end{equation*}
It is standard to prove the following result:

\begin{Lem}\label{lem8-27-4}
For any $u\in \tilde{E}$,  there exists $\bar\gamma>0$ such that
\begin{equation*}
\|\tilde  L(u)\|\geq \bar\gamma \|u\|,
\end{equation*}
where $\tilde  L$ is defined by
\begin{equation*}
\tilde  L(u):=-\Delta u(x)+ u(x)-3Q^2(x)u(x)+\frac{2}{a_*}Q(x)\int_{\R^N}Q^3(x) u(x).
\end{equation*}
\end{Lem}

\begin{Prop}\label{p-1-12-11}
Let $\tilde{\xi}_{a,i}(x)$ be as in \eqref{aaaaaluo--1}, then
\begin{equation}\label{8-27-2}
 \|\tilde{\xi}_{a,i}\| =O(\delta_a^4), ~\mbox{for}~i=1,\cdots,k.
\end{equation}
\end{Prop}
\begin{proof}
First, Lemma \ref{lem8-27-4} gives
\begin{equation}\label{8-28-21}
\|\tilde{\xi}_{a,i}\|\le C \|\tilde  L(\tilde{\xi}_{a,i})\|.
\end{equation}
On the other hand,  from \eqref{8-28-11}--\eqref{aaaaaluo--1}, we can prove

\begin{equation}\label{8-28-26}
\begin{split}
 \tilde  L(\tilde{\xi}_{a,i})& = O\big(\delta_a^4\big)
 +O\Big(\sum^2_{l=1}Q(x)
 v_a^{(l)}(\delta_a x+x^{(1)}_{a,i})
 \Big).
\end{split}\end{equation}
So from \eqref{8-27-26}, \eqref{8-28-21} and \eqref{8-28-26}, we prove \eqref{8-27-2}.

\end{proof}

\begin{Lem}
For $N=2,3$, we have the following  estimate on $\xi_a$:
\begin{equation}\label{8-19-10}
\begin{split}
  \delta_a^2\sum^k_{i=1}&  \int_{B_d(x^{(1)}_{a,i})}  \big(2V(x)+\langle \nabla V(x), x-x^{(1)}_{a,i}\rangle \big) (u_a^{(1)}+u_a^{(2)})\xi_a
\\=&
2\bigl( \mu^{(2)}_a -\mu_a^{(1)}\bigr)\delta_a^2\int_{\R^N}  u_a^{(1)}\xi_a
 + \int_{\R^N} (u_a^{(1)}+u_a^{(2)})\big(u_a^{(1)}-u_a^{(2)}\big)^2\xi_a 
\\&+
 \big(1-\frac{N}{2}\big)\sum^k_{i=1}\int_{B_d(x^{(1)}_{a,i})} (u_a^{(1)}+u_a^{(2)})\big((u_a^{(1)})^2
+(u_a^{(2)})^2\big)\xi_a+
O\Big(e^{-\frac{\gamma}{\delta_a}}\Big).
\end{split}
\end{equation}
\end{Lem}
\begin{proof}
Since $u_a^{(1)}$ and $u_a^{(2)}$ are two $k$-peak solutions of \eqref{1-12-11}--\eqref{a1-12-11}, then  similar to  \eqref{8-19-3},  we have following local Pohozaev identities:
 \begin{equation}\label{06-09-30}
 \begin{split}
 \int_{B_d(x^{(1)}_{a,i})} &\delta_a^2\big(2V(x)+\langle \nabla V(x), x-x^{(1)}_{a,i}\rangle \big)  (u_a^{(l)})^2
\\=& \int_{B_d(x^{(1)}_{a,i})} \big[2\mu_a^{(l)} \delta_a^2(u_a^{(l)})^2+(2-\frac{N}{2})(u_a^{(l)})^4\big]
+\delta_a^2\int_{\partial B_d(x^{(1)}_{a,i})}W^{(l)}(x)d\sigma,
\end{split}\end{equation}
where
\begin{equation*}
\begin{split}
W^{(l)}(x)=&-N\frac{\partial u^{(l)}_a}{\partial\nu_i}\langle x-x^{(1)}_{a,i},\nabla u^{(l)}_{a}\rangle
+\langle x-x^{(1)}_{a,i},\nu_i\rangle  |\nabla u^{(l)}_{a}|^2\\&
+\langle x-x^{(1)}_{a,i},\nu_i\rangle  \big[\big(V(x)-
\mu_a^{(l)}\big)\big(u^{(l)}_{a}\big)^2
-\frac{1}{2\delta_a^2}\big(u^{(l)}_{a}\big)^4\big].
\end{split}
\end{equation*}
Then \eqref{06-09-30} implies
 \begin{equation}\label{06-10-3}
 \begin{split}
 \int_{B_d(x^{(1)}_{a,i})} &\delta_a^2\big(2V(x)+\langle \nabla V(x), x-x^{(1)}_{a,i}\rangle \big)  (u_a^{(1)}+u_a^{(2)})\xi_a
\\=&\frac{2\delta_a^2 (\mu_{a}^{(1)}-\mu_{a}^{(2)})}
{\|u_{a}^{(1)}-u_{a}^{(2)}\|_{L^{\infty}(\R^N)}}\int_{B_d(x^{(1)}_{a,i})}\big(u_a^{(1)}\big)^2
-2\mu_{a}^{(2)} \delta_a^2
 \int_{B_d(x^{(1)}_{a,i})} (u_a^{(1)}+u_a^{(2)})\xi_a\\&
 + \big(2-\frac{N}{2}\big)\int_{B_d(x^{(1)}_{a,i})} (u_a^{(1)}+u_a^{(2)})\big((u_a^{(1)})^2
+(u_a^{(2)})^2\big)\xi_a
+ J_{a,i},
\end{split}\end{equation}
where
$J_{a,i}:=
\frac{ \delta_a^2}
{\|u_{a}^{(1)}-u_{a}^{(2)}\|_{L^{\infty}(\R^N)}}\displaystyle\int_{\partial B_d(x^{(1)}_{a,i})}\big(W^{(1)}(x)-W^{(2)}(x)\big)d\sigma.$

Next, we calculate the term $J_{a,i}$.
\begin{equation*}
\begin{split}
J_{a,i}=& -2\delta_a^2 \int_{\partial B_d(x^{(1)}_{a,i})}
\big[\frac{\partial u^{(1)}_a}{\partial\nu_i}\langle x-x^{(1)}_{a,i},\nabla \xi_{a}\rangle+
\frac{\partial \xi_a}{\partial\nu_i}\langle x-x^{(1)}_{a,i},\nabla u^{(1)}_a\rangle\big]\\&
+2\delta_a^2 \int_{\partial B_d(x^{(1)}_{a,i})}\langle x-x^{(1)}_{a,i},\nu_i\rangle \big[
 \nabla (u^{(1)}_{a}+u^{(2)}_{a}) \cdot  \nabla \xi_a+\big(V(x)-
\mu_a^{(1)} \big)
 (u^{(1)}_{a}+u^{(2)}_{a}) \xi_a  \big]
\\&
+\frac{\delta_a^2 (\mu_{a}^{(1)}-\mu_{a}^{(2)})}
{\|u_{a}^{(1)}-u_{a}^{(2)}\|_{L^{\infty}(\R^N)}}\int_{\partial B_d(x^{(1)}_{a,i})}\langle x-x^{(1)}_{a,i},\nabla u^{(1)}_a\rangle\big(u_a^{(2)}\big)^2\\&
-\frac{1}{2}\int_{\partial B_d(x^{(1)}_{a,i})}\langle x-x^{(1)}_{a,i},\nabla u^{(1)}_a\rangle(u_a^{(1)}+u_a^{(2)})\big((u_a^{(1)})^2
+(u_a^{(2)})^2\big)\xi_a\\=&
O\Big(e^{-\frac{\gamma}{\delta_a}}\Big).
\end{split}\end{equation*}
Summing  \eqref{06-10-3} from $i=1$ to $i=k$ and using \eqref{2--5}, \eqref{a2---5}, we find
 \begin{equation}\label{06-10--3}
 \begin{split}
 \sum^k_{i=1}&\int_{B_d(x^{(1)}_{a,i})}\delta_a^2\big(2V(x)+\langle \nabla V(x), x-x^{(1)}_{a,i}\rangle \big)  (u_a^{(1)}+u_a^{(2)})\xi_a
\\=&\big(2-\frac{N}{2}\big)\int_{\R^N} (u_a^{(1)}+u_a^{(2)})\big((u_a^{(1)})^2
+(u_a^{(2)})^2\big)\xi_a\\&+\frac{2\delta_a^2 (\mu_{a}^{(1)}-\mu_{a}^{(2)})}
{\|u_{a}^{(1)}-u_{a}^{(2)}\|_{L^{\infty}(\R^N)}}\int_{\R^N}\big(u_a^{(1)}\big)^2
 +
O\Big(e^{-\frac{\gamma}{\delta_a}}\Big).
\end{split}\end{equation}
Then from \eqref{06-10-2}  and  \eqref{06-10--3}, we deduce \eqref{8-19-10}.
\end{proof}

Let $\gamma_{a,i,j}$ be as in \eqref{aaaaaluo--1}, using $| \bar \xi_{a,i}|\le 1$, we find
\begin{equation}
\gamma_{a,i,j}= \frac{
   \bigl\langle  \bar \xi_{a,i}, \varphi_j\bigr\rangle}{\|\varphi_j\|^2} =O\bigl(\|\bar \xi_{a,i}\|\bigr)= O(1),~ j=0,1,\cdots,N.
\end{equation}

\begin{Lem}
For $N=2,3$, it holds
\begin{equation}\label{8-28-44}
\gamma_{a,i,0}=o(1), ~\mbox{for}~i=1,\cdots,k.
\end{equation}
\end{Lem}
\begin{proof}
First, we have
\begin{equation}\label{9-9-1}
u_a^{(2)}(\delta_a x+x^{(1)}_{a,i})=
 Q(x)
+O\Big(\frac{|x^{(1)}_{a,i}-x^{(2)}_{a,i}|}{\delta_a}|\nabla Q(x)|\Big)
+v_a^{(2)}(\delta_a x+x^{(1)}_{a,i}).
\end{equation}
Then from \eqref{8-27-38},  \eqref{8-27-26}, \eqref{8-20-1} and \eqref{9-9-1}, we know
\begin{equation}\label{aaa8-28-41}
\begin{split}
\int_{B_d(x^{(1)}_{a,i})}&\big(2V(x)+\langle \nabla V(x), x-x^{(1)}_{a,i}\rangle \big) (u_a^{(1)}+u_a^{(2)})\xi_a
\\
=&2\int_{B_d(x^{(1)}_{a,i})}\big(2(V(x)-V_i)+\langle \nabla V(x), x-x^{(1)}_{a,i}\rangle \big)
 Q_{\delta_a,x_{a,i}^{(1)}} \xi_{a}\\&
 +4V_i\int_{B_d(x^{(1)}_{a,i})}  (u_a^{(1)}+u_a^{(2)}) \xi_{a}
 +O\big(\delta_a^{N+4}\big).
\end{split}
\end{equation}
Also, from \eqref{8-27-38}, \eqref{aaaaaluo--1} and  \eqref{8-27-2}, we find
\begin{equation}\label{ad8-28-41}
\begin{split}
\int_{B_d(x^{(1)}_{a,i})}&\big[V(x)-V_i\big] Q_{\delta_a,x_{a,i}^{(1)}}\xi_{a}
\\=&
\delta_a^N\int_{\R^N}\big[V(\delta_a x+x_{a,i}^{(1)})-V_i\big]
 Q(x)\Big(
 \sum^N_{j=0}\gamma_{a,i,j}\psi_j\Big)+O\big(\delta_a^{N+4}\big)\\
 =&
-\frac{B}{2}\Delta V(b_i)\gamma_{a,i,0}\delta^{N+2}_a +O\big(\delta_a^{N+3}\big),
\end{split}
\end{equation}
where $B$ is the constant in \eqref{luopeng10}.
Similar to the estimate of \eqref{ad8-28-41}, we can find
\begin{equation}\label{9-9-22}
\begin{split}
\int_{B_d(x^{(1)}_{a,i})}&\langle \nabla V(x), x-x_{a,i}^{(1)}\rangle Q_{\delta_a,x_{a,i}^{(1)}}\xi_{a}
=-\frac{B}{2}\Delta V(b_i)\gamma_{a,i,0}\delta^{N+2}_a +O(\delta_a^{N+3}\big).
 \end{split}
\end{equation}
So from \eqref{ad8-28-41} and \eqref{9-9-22}, we get
\begin{equation}\label{8-28-41}
\begin{split}
\mbox{LHS of}~ \eqref{8-19-10}=-\frac{3B}{2}\Delta V(b_i)\gamma_{a,i,0}\delta^{N+4}_a
+O(\delta_a^{N+5})+
4V_i\delta^{2}_a\int_{B_d(x^{(1)}_{a,i})}  (u_a^{(1)}+u_a^{(2)}) \xi_{a}.
\end{split}
\end{equation}
Next we know 
\begin{equation}\label{8-28-41}
\begin{split}
 \int_{B_d(x^{(1)}_{a,i})}& (u_a^{(1)}+u_a^{(2)})\big((u_a^{(1)})^2
+(u_a^{(2)})^2\big)\xi_a\\
=&\Big(\frac{4-N}{4}+o(1)\Big)\gamma_{a,i,0}\delta^{N}_a\int_{\R^N}Q^4=\big(4a_*+o(1)\big)\gamma_{a,i,0}\delta^{N}_a.
\end{split}
\end{equation}
Also from \eqref{9-9-1}, we get
 \begin{equation*}
\begin{split}
\int_{B_d(x^{(1)}_{a,i})} & (u_a^{(1)}+u_a^{(2)}) \xi_{a}
\\=& 2\gamma_{a,i,0}\delta^{N}_a \int_{\R^N}Q(Q+x\cdot \nabla Q)+O\big(|x^{(1)}_a-x^{(2)}_a|\delta_a^{N-1} +\delta_a^{\frac{N}{2}} \|v_a^{(2)}\|_{\delta_a}\big)\\=&
\big(2-N\big)a_*\gamma_{a,i,0}\delta^{N}_a +
 O\big(\delta_a^{N+3}\big),
 \end{split}
\end{equation*}
which, together with  \eqref{8-28-41},  gives
\begin{equation}\label{a8-28-41}
\begin{split}
\mbox{LHS of}~ \eqref{8-19-10}=
2\big(2-N\big)\big(a_*+o(1)\big)\gamma_{a,i,0}\delta^{N}_a -\frac{3B}{2}\Delta V(b_i)\gamma_{a,i,0}\delta^{N+4}_a
+O(\delta_a^{N+5}).
\end{split}
\end{equation}
Also by \eqref{abc8-29-1}, \eqref{8-27-38} and \eqref{8-27-26}, we find
\begin{equation}\label{8-28-42}
\begin{split}
(\mu_a^{(2)}-\mu_a^{(1)})\int_{\R^N}u_a^{(2)}\xi_a
=O\big(\delta_a^{N+4}\big).
\end{split}
\end{equation}
On the other hand, by \eqref{8-27-26}, \eqref{8-20-1}, \eqref{a2---5}, \eqref{aaaaaluo--1}, \eqref{8-27-2}, \eqref{9-9-1} and \eqref{aaa8-28-41}, we find
\begin{equation}\label{8-28-43}
\begin{split}
\int_{\R^N} &(u_a^{(1)}+u_a^{(2)})(u_a^{(1)}-u_a^{(2)})^2\xi_a
\\=&\delta_a^N O\big(\delta_a^{-2}|x^{(1)}_a-x^{(2)}_a|^2 \big)+O\big( \|v^{(1)}_{a}-v^{(2)}_{a}\|_{\delta_a}^2\big)
=O\big(\delta_a^{N+6}\big).
\end{split}
\end{equation}
Then \eqref{8-19-10}, \eqref{a8-28-41}--\eqref{8-28-43} and Lemma \ref{lem-bbb} imply
\begin{equation*} 
\Big(4\big(2-N\big)\big(a_*+o(1)\big) - {3B} \delta^{4}_a\big(\sum^k_{l=1} \Delta V(b_l)\big)\Big)\gamma_{a,i,0}
=O(\delta_a^{5}),~\mbox{for}~i=1,\cdots,k,
\end{equation*}
which gives \eqref{8-28-44}.
\end{proof}

\begin{Prop}\label{p-2-12-11}
It holds
\begin{equation}\label{add8-28-44}
\gamma_{a,i,j}=o(1), \quad i=1,\cdots,k,~~j=0,1,\cdots,N.
\end{equation}
\end{Prop}
\begin{proof}
\textbf{Step 1:} To prove $\gamma_{a,i,N}=O(\delta_a)$ for $i=1,\cdots,k$.

\smallskip

Using \eqref{30-31-7}, we deduce
\begin{equation}\label{3-18}
\displaystyle \int_{B_{d}(x_{a,i}^{(1)})}\frac{\partial V(x)}{\partial  \nu_{a,i}}B_a(x)\xi_{a}=
O\big(e^{-\frac{\gamma}{\delta_a}}\big),
\end{equation}
where $\nu_{a,i}$ is the outward unit vector of $\partial B_{d}(x_{a,i}^{(1)})$ at $x$, $
B_a(x)=\displaystyle\sum^2_{l=1}u_{a}^{(l)}(x)$.

On the other hand, by \eqref{8-27-38}, we have
\begin{equation}\label{0-16-8}
\begin{split}
B_a(x)&=\Big(2 +O(\delta_a^2)\Big)
\sum^k_{i=1}Q_{\delta_a,x_{a,i}^{(1)}}(x)
+O\Big(\sum^{2}_{l=1}|v_{a}^{(l)}(x)|\Big).
\end{split}
\end{equation}
Also, from \eqref{abc8-29-1}, we find
\begin{eqnarray*}
\frac{\partial V(x_{a,i}^{(1)})}{\partial  \nu_{a,i}}=\frac{\partial V(x_{a,i}^{(1)})}{\partial  \nu_{a,i}}-
\frac{\partial V(\bar x_{a,i}^{(1)})}{\partial  \nu_{a,i}}= O\Big(\big|x_{a,i}^{(1)}-\bar x_{a,i}^{(1)}\big|\Big)=O(\delta_a^2),
\end{eqnarray*}
and
\begin{eqnarray*}
\frac{\partial^2 V(x_{a,i}^{(1)})}{\partial  \nu_{a,i}\partial  \tau_{a,i,j}}=\frac{\partial^2 V(x_{a,i}^{(1)})}{\partial  \nu_{a,i}\partial  \tau_{a,i,j}}-
\frac{\partial^2 V(\bar x_{a,i}^{(1)})}{\partial  \nu_{a,i}\partial  \tau_{a,i,j}}= O\Big(\big|x_{a,i}^{(1)}-\bar x_{a,i}^{(1)}\big|\Big)=O(\delta_a^2),~\mbox{for}~j=1,\cdots,N-1.
\end{eqnarray*}
From \eqref{2--5}, \eqref{30-31-7} and \eqref{0-16-8},  we get
\begin{equation}\label{3--18}
\begin{split}
 \displaystyle \int_{B_{d}(x_{a,i}^{(1)})}&\frac{\partial V(x)}{\partial  \nu_{a,i}}B_a(x)\xi_{a} \\=&
\int_{\R^N}  \frac{\partial V(x_{a,i}^{(1)})}{\partial  \nu_{a,i}} B_a(x)\xi_{a}+\int_{\R^N} \Bigl\langle \nabla \frac{\partial V(x_{a,i}^{(1)})}{\partial  \nu_{a,i}}, x-x_{a,i}^{(1)}\Bigr\rangle B_a(x)\xi_{a}
 +O\big(\delta_a^{N+2}\big)\\=&
-\frac{\partial^2 V(x_{a,i}^{(1)})}{\partial  \nu^2_{a,i}} a_*\gamma_{a,i,N} \delta_a^{N+1}+O\big(\delta_a^{N+2}\big).
\end{split}
\end{equation}
Then \eqref{3-18} and \eqref{3--18} imply  $\gamma_{a,i,N}=O(\delta_a)$.

\medskip

\noindent\textbf{Step 2:} To prove  $\gamma_{a,i,j}=o(1)$ for $i=1,\cdots,k$ and $j=1,\cdots,N-1$.

\smallskip

Similar to \eqref{3-18}, we have
\begin{equation}\label{3.-14}
 \int_{B_{d}(x_{a,i}^{(1)})} \frac{\partial V(y)}{\partial \tau_{a,i,j}}B_a(y)\xi_{a}=O\big(e^{-\frac{\gamma}{\delta_a}}\big),
 ~\mbox{for}~i=1,\cdots,k,~~j=1,\cdots,N-1.
 \end{equation}
Using suitable rotation, we assume that $\tau_{a,i,1} =(1, 0,\cdots, 0), \cdots$, $\tau_{a,i, N-1} =(0,\cdots,0, 1,  0)$ and $\nu_{a,i} =(0,\cdots,0,1)$.
Under the condition \textup{($V$)}, we know
\begin{equation}\label{1-16-8}
\begin{split}
 \frac{\partial V(\delta_a y+ x^{(1)}_{a,i})}{\partial \tau_{a,i,j}}=&
\delta_a\sum^N_{l=1}\frac{\partial^2 V(x^{(1)}_{a,i})}{\partial y_l \partial \tau_{a,i,j} } y_l+\frac{\delta_a^2}{2}
\sum^N_{m=1}\sum^N_{l=1}\frac{\partial^3 V(x^{(1)}_{a,i})}{\partial y_m  \partial y_l \partial \tau_{a,i,j}} y_m y_l\\&+
\frac{\delta^3_a}{6}\sum^N_{s=1}\sum^N_{m=1}
\sum^N_{l=1}\frac{\partial^4 V(x^{(1)}_{a,i})}{\partial y_s \partial y_l \partial y_m \partial \tau_{a,i,j}} y_s y_ly_m
+o\big(\delta_a^3|y|^3\big),~\mbox{in}~B_{d{\delta_a^{-1}}}(0).
\end{split}\end{equation}
By \eqref{1.6}, \eqref{8-22-3}, \eqref{8-27-26}, \eqref{aaaaaluo--1}, \eqref{0-16-8} and  the symmetry of $\varphi_j(x)$, we find, for $j=1,\cdots,N-1$,
\begin{equation}\label{2-16-8}
\begin{split}
\sum^N_{m=1}\sum^N_{l=1}&\frac{\partial^3 V(x^{(1)}_{a,i})}{\partial y_m  \partial y_l \partial \tau_{a,i,j}}\int_{B_{d\delta_a^{-1}}(0)} B_a(\delta_a y+ x^{(1)}_{a,i}) \bar \xi_{a,i} y_m y_l\\=&
2 \sum^N_{m=1}\sum^N_{l=1}\frac{\partial^3 V(x^{(1)}_{a,i})}{\partial y_m  \partial y_l \partial \tau_{a,i,j}} \int_{B_{d\delta_a^{-1}}(0)} Q\big(\sqrt{1+V_i\delta_a^2}y\big)\bar \xi_{a,i} y_m y_l+O(\delta_a^2)
\\
=& B\gamma_{a,i,0} \frac{\partial \Delta V(x^{(1)}_{a,i})}{\partial \tau_{a,i,j}} +O(\delta_a^2)=
O\big(|x^{(1)}_{a,i}-b_i|\big) +O(\delta_a^2)=O(\delta_a^2).
\end{split}
\end{equation}
Also from \eqref{add8-28-44} and \eqref{0-16-8}, we get
\begin{equation}\label{3-16-8}
\begin{split}
& \sum^N_{s=1}\sum^N_{m=1}\sum^N_{l=1}\frac{\partial^4 V(x^{(1)}_{a,i})}{\partial y_s \partial y_l \partial y_m \partial \tau_{a,i,j}}
 \int_{B_{d\delta_a^{-1}}(0)} B_a(\delta_a y+ x^{(1)}_{a,i}) \bar \xi_{a,i}y_s y_ly_m
\\=&
2 \sum^N_{s=1}\sum^N_{m=1}\sum^N_{l=1}\frac{\partial^4 V(x^{(1)}_{a,i})}{\partial y_s \partial y_l \partial y_m \partial \tau_{a,i,j}}
 \int_{B_{d\delta_a^{-1}}(0)}  Q(y)(\sum^{N-1}_{q=1}\gamma_{a,i,q})\varphi_q(y)y_s y_ly_m
+o\big(1\big)\\=&
2 \sum^{N-1}_{q=1}\gamma_{a,i,q}
 \int_{B_{d\delta_a^{-1}}(0)} Q(y) \varphi_q(y)y_q \Big[\frac{\partial^4 V(x^{(1)}_{a,i})}{\partial y^3_q \partial \tau_{a,i,j}}y_q^2
 +3 \sum^N_{l=1,l\neq q}\frac{\partial^4 V(x^{(1)}_{a,i})}{\partial y_q \partial y^2_l \partial \tau_{a, i,j}} y_l^2\Big]+o\big(1\big)\\=&
-3B\Big(\sum_{q=1}^{N-1}\frac{\partial^2 \Delta V(x^{(1)}_{a,i})}{\partial \tau_{a, i,q} \partial \tau_{a,i, j}}  \gamma_{a,i,q}\Big)+o\big(1\big)\\=&
-3B\Big(\sum_{q=1}^{N-1}\frac{\partial^2 \Delta V(b_i)}{\partial \tau_{a, i,q} \partial \tau_{a,i, j}}  \gamma_{a,i,q}\Big)+o\big(1\big).
\end{split}
\end{equation}
By \eqref{ab7-19-29},  we estimate
\[
\frac{\partial^2 V(x^{(1)}_{a,i})}{\partial y_l \partial \tau_{a,i, j} }=  - \frac{\partial V(x^{(1)}_{a,i})}{ \partial \nu_{a,i} }
\kappa_{i,l}(x^{(1)}_{a,i})\delta_{lj},\quad l,j=1,\cdots,N-1.
\]
 Since  $\frac{\partial V(\bar x^{(1)}_{a,i})}{ \partial \nu_{a,i} }=0$, from \eqref{8-22-3}, we find
\begin{equation}\label{lll}
\begin{split}
\frac{\partial^2 V(x^{(1)}_{a,i})}{\partial y_l \partial \tau_{a,i,j} }=&  -\Bigl(  \frac{\partial V(x^{(1)}_{a,i})}{ \partial \nu_{a,i} }-
\frac{\partial V(\bar x^{(1)}_{a,i})}{ \partial \nu_{a,i} }
\Bigr)
\kappa_{i,l} (x^{(1)}_{a,i})\delta_{lj} \\
=&- \frac{\partial^2 V(\bar x^{(1)}_{a,i})}{ \partial \nu^2_{a,i} } ( x^{(1)}_{a,i}-\bar x^{(1)}_{a,i})\cdot \nu_{a,i}\kappa_{i,l} (x^{(1)}_{a,i})\delta_{lj}+o(\delta_a^2)\\
=&- \frac{\partial^2 V(b_i)}{ \partial \nu^2_i } ( x^{(1)}_{a,i}-\bar x^{(1)}_{a,i})\cdot \nu_{a,i}\kappa_{i,l} (b_i)\delta_{lj}+o(\delta_a^2)\\
=&\frac{B}{2a_*}\frac{\partial  \Delta V(b_i)}{\partial \nu_i}  \delta_a^{2}\kappa_{i,l}(b_i)\delta_{lj}+o(\delta_a^2).
\end{split}
\end{equation}
Therefore from \eqref{8-27-26}, \eqref{aaaaaluo--1}, \eqref{0-16-8} and \eqref{lll}, we get
\begin{equation}\label{5-16-8}
\begin{split}
\sum^N_{l=1} &\frac{\partial^2 V(x^{(1)}_{a,i})}{\partial y_l \partial \tau_{a,i,j} }\int_{B_{d\delta_a^{-1}}(0)} B_a(\delta_a y+ x^{(1)}_{a,i}) \bar \xi_{a,i} y_l\\
=& \frac{B}{2a_*}\frac{\partial  \Delta V(b_i)}{\partial \nu_i}  \delta_a^{2}\kappa_{i,j} (b_i)\int_{B_{d\delta_a^{-1}}(0)} B_a(\delta_a y+ x^{(1)}_{a,i}) \bar \xi_{a,i} y_j
+o(\delta_a^2)\\
=& \frac{B}{a_*}\frac{\partial  \Delta V(b_i)}{\partial \nu_i}  \delta_a^{2}\kappa_{i,j}(b_i)\gamma_{a,i,j} \int_{\mathbb R^N }Q(y)
 \frac{Q(y)}{\partial y_j} y_j
+o(\delta_a^2)\\
=& -\frac{B}{2}\frac{\partial  \Delta V(b_i)}{\partial \nu_i}  \delta_a^{2}\kappa_{i,j}(b_i)\gamma_{a,i,j}
+o(\delta_a^2).
\end{split}\end{equation}
Combining \eqref{1-16-8}--\eqref{5-16-8}, we obtain
\begin{equation}\label{6-16-8}
\begin{split}
\int_{B_{d}(x_{a,i}^{(1)})} & \frac{\partial V(y)}{\partial \tau_{a,i,j}}B_a(y)\xi_{a}\\=&-\frac{B}{2}\frac{\partial  \Delta V(b_i)}{\partial \nu_i}  \delta_a^{N+3}\kappa_{i,j}(b_i)\gamma_{a,i,j}
-\frac{B}{2}\Big(\sum_{l=1}^{N-1}\frac{\partial^2 \Delta V(b_i)}{\partial \tau_{i,l} \partial \tau_{ i,j}} \gamma_{a,i,j} \Big) \delta_a^{N+3}+o(\delta_a^{N+3}).
 \end{split}\end{equation}
 From \eqref{3.-14} and \eqref{6-16-8}, we find
\[
\frac{\partial  \Delta V(b_i)}{\partial \nu_i} \kappa_{i,j}(b_i)\gamma_{a,i,j}
+\Big(\sum_{l=1}^{N-1}\frac{\partial^2 \Delta V(b_i)}{\partial \tau_{i,l} \partial \tau_{i,j}} \gamma_{a,i,l}\Big) =o(1),
\]
which implies $\gamma_{a,i,j}=o(1)$ for $i=1,\cdots,k$ and $j=1,\cdots,N-1$.

\end{proof}

\begin{proof}[\textbf{Proof of Theorem \ref{th1.4}:}]
First, for large fixed $R$,  \eqref{8-28-12} and \eqref{8-28-13} give
$$
C_{a}(x) \ge \frac12,   \;\; |g_a(x)|\le C \sum^k_{i=1}e^{\frac{|x-x^{(1)}_{a,i}|}{\delta_a}},\quad x\in \mathbb R^N \backslash  \bigcup ^k_{i=1}B_{R\delta_a}(x^{(1)}_{a,i}).$$
  Using the comparison principle, we get
\begin{equation*}
 \xi_a(x)= o(1),~\mbox{in}~ \R^N\backslash  \bigcup ^k_{i=1}B_{R\delta_a}(x^{(1)}_{a,i}).
\end{equation*}

On the other hand, it follows from \eqref{8-27-2}, \eqref{8-28-44} and \eqref{add8-28-44} that
\[
\xi_{a}(x)=o(1),~\mbox{in}~ \bigcup ^k_{i=1}B_{R\delta_a}(x^{(1)}_{a,i}).
\]
This is in contradiction with $\|\xi_{a}\|_{L^{\infty}(\R^N)}=1$.
So $u^{(1)}_{a}(x)\equiv u^{(2)}_{a}(x)$ for $a\rightarrow ka_*$ in $N=2$ or $a\searrow 0$ in $N=3$.
\end{proof}

\section*{Appendix}

\appendix

\section{Linearization}

\renewcommand{\theequation}{A.\arabic{equation}}
\begin{Lem}\label{lem-bbb}
Let $\xi_0:=(\xi_1,\cdots,\xi_k)$ be the solution of following system:
\begin{equation}\label{06-10-23}
-\Delta \xi_{i}(x)+\big(1-3Q^2(x)\big)\xi_{i}(x)=-\frac{2}{ka_*}Q(x)
\Big(\sum^k_{l=1}\int_{\R^N}Q^3(x)\xi_l(x)\Big),~\mbox{for}~i=1,\cdots,k.
\end{equation}
Then for $N=2,3$, it holds
\begin{equation}\label{06-10-25}
\xi_i(x)=  \sum^N_{j=0}\gamma_{i,j}\psi_j,
\end{equation}
where  $\gamma_{i,j}$ are some constants,
\begin{equation}\label{aaaaa}
\psi_0=Q+x\cdot\nabla Q,~\psi_j=\frac{\partial Q}{\partial x_j},~\mbox{for}~j=1,\cdots,N.
\end{equation}
 Moreover, $\gamma_{i,0}=\gamma_{l,0}$ for all $i,l=1,\cdots,k$.
\end{Lem}
\begin{proof}
We set $\bar L(u):=-\Delta u(x)+\big(1-3Q^2(x)\big)u(x)$. Let $\Psi_j=(\psi_j,\cdots,\psi_j)$, it is obvious that $\bar  L (\Psi_j)=0$ and $\Psi_j$ is the solution of \eqref{06-10-23}, for $j=1,\cdots,N$.
Also, let $$\Phi= \big(Q+x\cdot\nabla Q,\cdots,Q+x\cdot\nabla Q\big),$$
then using \eqref{06-13-1}, we find that $\Phi$ is also the solution of \eqref{06-10-23}. And we know that $\Phi,\Psi_1,\cdots,\Psi_N$ are linearly independent. Then we get \eqref{06-10-25}.

Moreover putting \eqref{06-10-25} into \eqref{06-10-23}, we find
\begin{equation*}
-2Q(x)\gamma_{i,0}=-\frac{2}{k}Q(x)
\sum^k_{l=1}\gamma_{l,0},
\end{equation*}
which gives $\gamma_{i,0}=\gamma_{l,0}$ for all $i,l=1,\cdots,k$.
\end{proof}

\section{Calculations involving curvatures}
\renewcommand{\theequation}{B.\arabic{equation}}

\setcounter{equation}{0}

Now let $\Gamma\in C^2$ be a closed hypersurface in $\mathbb R^N$. For $y\in \Gamma$, let $\nu(y)$ and $T(y)$ denote respectively the outward unit  normal to $\Gamma$ at $y$ and the tangent hyperplane to $\Gamma$ at $y$. The curvatures of $\Gamma$ at a fixed point $y_0\in\Gamma$ are determined as follows. By a rotation of coordinates, we can assume that
$y_0=0$ and $\nu(0)$ is the $x_N$-direction, and $x_j$-direction is the $j$-th principal direction.

In some neighborhood $\mathcal{N}=\mathcal{N}(0)$ of $0$,   we have
\[
\Gamma=\bigl\{  x:   x_N=\varphi(x')\bigr\},
\]
where $x'=(x_1,\cdots,x_{N-1})$,
 \[
 \varphi(x') =\frac12 \sum_{j=1}^{N-1} \kappa_j x_j^2 +  O(|x'|^3),
 \]
 where
 $\kappa_j$, is  the $j$-th principal curvature of  $\Gamma$ at $0$.
 The Hessian matrix $[D^2 \varphi(0)]$  is given by
\begin{equation*}
[D^2 \varphi(0)]=diag [\kappa_1,\cdots,\kappa_{N-1}].
\end{equation*}

Suppose that  $W$ is a smooth function, such  that  $W(x)=a$  for all $x\in \Gamma$.

\begin{Lem}\label{lem-8-11-2}
We have
 \begin{equation}\label{7-19-29}
 \frac{\partial W(x)}{\partial x_l}\Bigr|_{x=0}=0, ~\mbox  ~l=1,\cdots,N-1,
 \end{equation}
  \begin{equation}\label{ab7-19-29}
  \frac{\partial^2 W(x)}{\partial x_m\partial x_l }\Bigr|_{x=0}=-\frac{\partial W\big(x\big)}{\partial x_N}\Bigr|_{x=0}\kappa_i \delta_{ml}, ~\mbox{for}~m, l=1,\cdots, N-1,
 \end{equation}
 where  $\kappa_1,\cdots,\kappa_{N-1}$, are  the principal curvatures of  $\Gamma$ at $0$.
\end{Lem}
\begin{proof}
First, we  have $
W\big(x',\varphi(x')\big)=0$.
And then we find
\begin{equation}\label{8-11-1}
 \frac{\partial W\big(x',\varphi(x')\big)}{\partial x_m}
 + \frac{\partial W\big(x',\varphi(x')\big)}{\partial x_N}\frac{\partial \varphi(x')}{\partial x_m} =0, ~\mbox{for}~m=1,\cdots,N-1.
 \end{equation}
Letting $x'=0$ in \eqref{8-11-1}, we obtain \eqref{7-19-29}.

Differentiating  \eqref{8-11-1} with respect to $x_l$ for $l=1,\cdots,N-1$, we get
 \begin{equation}\label{8-11-2}
 \begin{split}
 \frac{\partial^2  W\big(x',\varphi(x')\big)}{\partial x_m\partial x_l}&+\frac{\partial^2 W\big(x',\varphi(x')\big)}{\partial x_m\partial x_N}\frac{\partial \varphi(x')}{\partial x_l}
 + \frac{\partial W\big(x',\varphi(x')\big)}{\partial x_N}\frac{\partial^2 \varphi(x')}{\partial x_m x_l}
  \\&+ \Big(\frac{\partial^2 W \big(x',\varphi(x')\big)}{\partial x_N \partial x_l}+\frac{\partial^2 W\big(x',\varphi(x')\big)}{\partial x_N\partial x_N}\frac{\partial \varphi(x')}{\partial x_l}\Big) \frac{\partial \varphi(x')}{\partial x_m}=0.
 \end{split}
 \end{equation}
 Let $x'=0$ in \eqref{8-11-2}, then we get  \eqref{ab7-19-29}.
\end{proof}

\section{An example}
\renewcommand{\theequation}{C.\arabic{equation}}

\setcounter{equation}{0}
In this section, we use the above results to the following potential $V(x)$.
Let
\[
F_1(x)= \sum^N_{j=1}\frac{x_j^2}{a_j^2}-1,~F_2(x)= \sum^N_{j=1}\big(\frac{x_j}{a_j}-3\big)^2-1,
\]
where $a_j>0$, $a_j\ne a_l$ for $j\ne l$.   Let  $\Gamma_i$ is defined by  $F_i(x)=0$ with $i=1,2$.   Take
\begin{equation*}
V(x)=
\begin{cases}
F_1^2+1,~\mbox{in}~W_1,\\
F_2^2+1,~\mbox{in}~W_2,\\
\mbox{else},~\mbox{in}~\R^N\backslash \bigcup^2_{i=1}W_i.
\end{cases}
\end{equation*}
  and
$$W_i=\Big\{x\in\R^N; |F_i|\leq \delta_0\Big\},~\mbox{with some small fixed}~\delta_0>0~\mbox{and}~i=1,2.$$
\begin{Lem}
All critical points $\Delta V$ on $\Gamma_1$ are  $(\pm a_1,0,\cdots, 0), \cdots, (0,\cdots, 0, \pm a_N)$.
\end{Lem}
\begin{proof}
First, we find
\begin{equation*}
\nabla V(x) = 2 F \nabla F,~~\Delta V=  2 F \Delta F +2 |\nabla F|^2=  \Big(\sum^N_{l=1}\frac{x_l^2}{a_l^2}-1\Big) \sum_{l=1}^N\frac 4{a_i^2}+ \sum_{l=1}^N \frac{8x_l^2}{a_l^4}.
\end{equation*}
To find a critical point $\Delta V$ on $\Gamma$, we need to study the following equation
\[
\nabla (\Delta V) =\lambda  \nabla F,
\]
for some unknown constant $\lambda$.  That is,
\[
\frac{8x_l}{a_l^2} \sum_{k=1}^N \frac{1}{a_k^2} + \frac{16 x_i}{a_l^4} = \frac{2\lambda x_i}{a_l^2}, \quad l=1,\cdots, N.
\]
Thus, either $x_l=0$, or  $\lambda= 4 \displaystyle\sum_{k=1}^N \frac{1}{a_k^2} + \frac{8}{a_l^2}$.  If  $\lambda=4  \displaystyle\sum_{k=1}^N \frac{1}{a_k^2} + \frac{8}{a_l^2}$,
then $x_j=0$ for all $j\ne l$.  This shows that all critical points $\Delta V$ on $\Gamma$ are  $(\pm a_1,0,\cdots, 0), \cdots, (0,\cdots, 0, \pm a_N)$.
\end{proof}
Without loss of generality, we consider the point  $b_1=(0,\cdots, 0,  a_N)$. In this case,  $\tau_j$ is the $x_j$ direction, $j=1, \cdots, N-1$, and $\nu$
is the $x_N$ direction.
\begin{Lem}
If $a_l\ne a_j$ for $l\ne j$. Thus, $b_1$ is  non-degenerate on $\Gamma_1$.
\end{Lem}
\begin{proof}
We have
\begin{equation*}
\frac{\partial^2 V(b_1)}{\partial x_N^2}=\frac{8}{a_N^2}>0,~~
\frac{\partial  \Delta V(b_1)}{\partial x_N}=\frac{8}{a_N} \sum_{l=1}^N\frac {1}{a_l^2}+ \frac{16}{a_N^3}>0.
\end{equation*}
On $\Gamma_1$, it holds
\begin{equation*}
\Delta V(x)=\sum_{l=1}^N \frac{8 x_l^2}{a_l^4}= 8\Bigl( \sum_{l=1}^{N-1} \frac{ x_l^2}{a_l^4}+\frac1{a_N^2}\Bigl( 1- \sum_{l=1}^{N-1} \frac{ x_l^2}{a_l^2}\Bigr)
\Bigr).
\end{equation*}
This gives
\begin{equation*}
\Big(\frac{\partial^2 \Delta V(b_1)}{\partial x_lx_j}\Big)_{1\leq l,j\leq N-1}=  \;diag~\Bigl(\frac{16}{a_1^2}\Big( \frac1{a_1^2}-\frac1{a_N^2}\Big),\cdots,\frac{16}{a_{N-1}^2}\Big( \frac1{a_{N-1}^2}-\frac1{a_N^2}\Big)\Bigr),
\end{equation*}
which is non-singular since $a_l\ne a_j$ for $l\ne j$. Thus, $b_1$ is also non-degenerate on $\Gamma_1$.
\end{proof}
\begin{Lem}
The  matrix
\begin{equation}\label{06-07-5}
\Big(\frac{\partial^2 \Delta V(b_1)}{\partial x_ix_j}\Big)_{1\leq i,j\leq N-1}+\frac{\partial  \Delta V(b_1)}{\partial x_N}  diag (\kappa_1, \cdots, \kappa_{N-1})
\end{equation}
is non singular if one of the following conditions holds£º

\smallskip

\textup{(1)} $a_N<a_l$, $l=1,\cdots, N-1$.

\smallskip

\textup{(2)} $a_N>a_l$, $l=1,\cdots, N-1$ and all the $a_i$ are close to a constant.
\end{Lem}
\begin{proof}
Near $b_1=(0,\cdots, 0,  a_N)$, $\Gamma$ is given by
\[
x_N= a_N \sqrt{ 1-\sum_{l=1}^{N-1} \frac{x_l^2}{a_l^2}}=  a_N -\frac12 \sum_{l=1}^{N-1} \frac{ a_N x_l^2}{a_l^2}+ O(|x'|^3).
\]
Thus, $\kappa_l = -\frac{ a_N }{a_l^2},\quad l=1,\cdots, N-1$.
So we have
\begin{equation*}
\frac{\partial  \Delta V(b_1)}{\partial x_N} \kappa_j(b_1)=-\Bigl(\frac{8}{a_j^2} \sum_{l=1}^N\frac 1{a_l^2}+ \frac{16}{a_j^2a_N^2}\Bigr).
\end{equation*}

If  $x_0$ is a maximum point of $\Delta V$ on $\Gamma$, that is $a_N<a_l$, $l=1,\cdots, N-1$, then
$
\Big(\frac{\partial^2 \Delta V(x_0)}{\partial x_lx_j}\Big)_{1\leq l,j\leq N-1}$
is negative. Thus, \eqref{06-07-5}
is also a negative matrix.
On the other hand,  if  $b_1$ is a minimum point of $\Delta V$ on $\Gamma_1$
and all the $a_j$ are close to a constant, that is $a_N>a_l$, $l=1,\cdots, N-1$, then
 \eqref{06-07-5}
is negative.
\end{proof}

Finally, for $b_2=(0,\cdots,0,4a_N)\in \Gamma_2$, we have the similar results.


\end{document}